\documentclass{amsart}

\usepackage[T1]{fontenc}
\usepackage{mathrsfs}
\usepackage{amsmath}
\usepackage{amssymb}
\usepackage{amsthm}
\usepackage{epsfig}
\usepackage{hyperref}
\usepackage{subfig}
\input xy
\xyoption{all}

\newtheorem{theorem}{Theorem}[section]

\newtheorem{proposition}[theorem]{Proposition}
\theoremstyle{definition}
\newtheorem{definition}[theorem]{Definition}
\newtheorem{example}[theorem]{Example}
\newtheorem*{maintheorem}{Theorem \ref{thm:main}}

\newtheorem{remark}[theorem]{Remark}
\numberwithin{equation}{section}

\def\G{{\Gamma}}
\def\m{{MCG_*(\Sigma)}}
\def\A{{\mathbb{A}}}
\def\SA{{S\A}}
\def\T{{\mathbb{T}}}
\def\ST{{S\T}}

\begin{document}

\noindent                                             
\begin{picture}(150,36)                               
\put(5,20){\tiny{To appear in}}                       
\put(5,7){\textbf{Topology Proceedings}}
\put(0,0){\framebox(140,34){}}                      
\put(2,2){\framebox(136,30){}}                   
\end{picture}                                      

\vspace{0.5in}

\renewcommand{\bf}{\bfseries}
\renewcommand{\sc}{\scshape}

\vspace{0.5in}

\title[A Degree Theorem for the Space of Ribbon Graphs]
{A Degree Theorem for the Space of Ribbon Graphs}

\author{Bradley Forrest}
\address{101 Vera King Farris Drive; Galloway, NJ 08205}
\email{bradley.forrest@stockton.edu}


\begin{abstract}
We show that the space of ribbon graphs of an orientable, basepointed, genus $g$ surface $\Sigma$ with $p$ punctures can be filtered by simplicial complexes.  Specifically, the space of ribbon graphs contains $\ST_{\Sigma, 0} \subset \ST_{\Sigma, 1} \subset \ldots \subset \ST_{\Sigma, 4g+2p-4}$ where $\ST_{\Sigma, k}$ is a $k$-dimensional, $(k-1)$-connected simplicial complex that is invariant under the action of the basepoint preserving mapping class group of $\Sigma$.
\end{abstract}

\maketitle

\section{\bf Introduction}

Hatcher and Vogtmann developed Auter space in \cite{hatcher-vogtmann} as a means to study the homology of $Aut(F_n)$, the automorphism group of a finitely generated free group.  Auter space $\A_n$ is an analog, in the context of $Aut(F_n)$, of Culler and Vogtmann's Outer space \cite{culler-vogtmann}. More specifically, $\A_n$ is contractible and contains a simplicial spine $\SA_n$ that $Aut(F_n)$ acts on cocompactly and simplicially with finite simplex stabilizers.  Outer and Auter space are both unions of open simplicies defined by certain classes of graphs, however, in the latter case the graphs come equipped with basepoints.  Hatcher and Vogtmann took advantage of this technical difference to prove \emph{the degree theorem} \cite{hatcher-vogtmann}, that Auter space contains a sequence of simplicial complexes $\SA_{n,0} \subset \SA_{n,1} \subset \ldots \subset \SA_{n, 2n-2} = \SA_n$ indexed by the degree of the spanning graphs.  These spaces are known as \emph{the degree complexes} and have many convenient properties.  In particular, $\SA_{n, k}$ is $k$-dimensional, $(k-1)$-connected, and preserved under the action of $Aut(F_n)$.  In addition to powering Hatcher and Vogtmann's investigation of the homology of $Aut(F_n)$ in  \cite{hatcher-vogtmann}, the degree theorem was used by Armstrong, Forrest, and Vogtmann to compute a relatively elegant presentation of $Aut(F_n)$ in \cite{afv}.

The long standing analogy between automorphism groups of free groups and mapping class groups of orientable surfaces has driven many important results \cite{bestvina-handel} \cite{culler-vogtmann} \cite{kapovich}.  Traditionally, the analogy has been used to inform the study of $Aut(F_n)$, with techniques developed in the context of mapping class groups producing results for $Aut(F_n)$.  In this work, we use this relationship in the opposite direction to develop an analog of the degree theorem in the context of mapping class groups.  To fix our notation, let $\Sigma$ be an orientable genus $g$ surface with $p \geq 1$ punctures and a basepoint.  Further let $\m$ be the basepoint preserving mapping class group of $\Sigma$, the group of isotopy classes of orientation preserving homeomorphisms of $\Sigma$ in which both the isotopies and homeomorphisms preserve the basepoint.  Playing the roles of Auter space and its simplicial spine are the space $\T_{\Sigma}$ and complex $\ST_{\Sigma}$ of ribbon graphs that can be drawn in $\Sigma$.  Both of these spaces are related to the decorated Teichm\"uller space of $\Sigma$ which was developed by Penner in \cite{penner} and independently by Bowditch and Epstien in \cite{bowditch-epstein}.  Taking the basepoint as a puncture, $\T_{\Sigma}$ is the projectivized decorated Teichm\"uller space with decorated non-basepoint punctures \cite{penner-book}.  Decorated Teichm\"uller spaces and related spaces have played important roles in the study of orientable surfaces and their mapping class groups \cite{horak} \cite{penner} \cite{penner2} \cite{penner-book}.

Our main theorem is the following analog of Hatcher and Vogtmann's degree theorem:

\begin{theorem}
The space of basepointed ribbon graphs $\T_{\Sigma}$ contains simplicial complexes $\ST_{\Sigma, 0} \subset \ST_{\Sigma, 1} \subset \ldots \subset \ST_{\Sigma, 4g+2p-4} = \ST_{\Sigma}$.  The complex $\ST_{\Sigma, k}$ is $k$-dimensional, $(k-1)$-connected, and preserved under the action of $\m$.
\label{thm:main}
\end{theorem}

While this result is motivated by the successful applications of the degree theorem in the study of $Aut(F_n)$, applications of Theorem \ref{thm:main} have yet to be produced.  The complexes $\ST_{\Sigma, k}$ lack a key property that Hatcher and Vogtmann's degree complexes enjoy.  Specifically, the quotient spaces $\SA_{n, k} / Aut(F_n)$ stabilize for large $n$ while the quotients $\ST_{\Sigma, k} / \m$ grow unbounded as either the genus or the number of puncters of $\Sigma$ increases.

This paper is organized as follows.  Section 2 introduces the relevant background regarding the groups $\m$ and $Aut(F_n)$ and the topological spaces $\T_{\Sigma}$, $S\T_{\Sigma}$, $\A_n$ and $S\A_n$.  The overall idea of the proof is outlined in Section 3, and the primary methods that we use are introduced.  Section 4 adapts Hatcher and Vogtmann's degree theorem to the context of $\T_{\Sigma}$, proving Theorem \ref{thm:main}.  In Section 5, we take a closer look at the complex $S\T_{\Sigma, 2}$ and possible applications for and open questions regarding the simplicial complexes $S\T_{\Sigma, k}$.

\section{\bf Background}

In this section we review the basic definitions and topological constructions 
regarding the space of ribbon graphs and Auter space.  For a more detailed treatment, especially proofs of many facts claimed in this section, see \cite{culler-vogtmann}, \cite{hatcher-vogtmann}, \cite{horak},  \cite{mulase-penkava} and \cite{penner-book}.  All graphs and surfaces in the following definitions are basepointed.  All maps between surfaces and graphs are basepoint preserving.

\subsection{Auter Space}

Fix a positive integer $n$ and let $R_n$ be a graph with one vertex and $n$ edges.  A \emph{marking} on a finite connected graph $\G$ is a homotopy equivalence $\phi \colon R_{n} \rightarrow \Gamma$.  The marking on $\Gamma$ is most simply denoted by choosing a maximal forest in $\Gamma$, choosing orientations for the edges not in the maximal forest and labeling those edges by the elements of $\pi_1(R_n)$ that they map to under a homotopy inverse $\phi^{-1} \colon \Gamma \to R_n$.  This labeling is not canonical, it depends on the chosen maximal forest.  Two marked graphs $(\G_{1}, \phi_1)$ and $(\G_{2}, \phi_2)$ are \emph{equivalent} if there exists a graph isomorphism $h \colon \G_{1} \rightarrow \G_{2}$ so that $h \circ \phi_{1}$ is homotopic to  $\phi_{2}$.  A marked graph $(\G_{2}, \phi_{2})$ is obtained from $(\G_{1}, \phi_{1})$ by a forest collapse if there exists a quotient map $q \colon \G_{1} \rightarrow \G_{2}$ that collapses a forest in $\G_{1}$ to a disjoint union of vertices in $\G_{2}$, and is otherwise the identity map, that makes $q \circ \phi_{1}$ homotopic to $\phi_{2}$.  The set of equivalence classes of marked graphs is a poset where $(\G_{2},\phi_{2}) \leq (\G_{1},\phi_{1})$ if $(\G_{2},\phi_{2})$ can be obtained from $(\G_{1},\phi_{1})$ by a forest collapse.  Auter space is a union of open simplices where each open $i$-simplex is given by an equivalence class of marked graphs with $i+1$ edges for which all non-basepoint vertices are required to have valence at least 3 and the basepoint has valence at least 2.  These simplices glue together by the ordering described above, specifically the simplex corresponding to $(\G_{2},\phi_{2})$ is a face of the simplex corresponding to $(\G_{1},\phi_{1})$ if $(\G_{2},\phi_{2}) \leq (\G_{1},\phi_{1})$.  The open simplex for $(\G, \phi)$ can be constructed by endowing $\G$ with a normalized metric, enforcing that the sum of the lengths of the edges of $\G$ is 1.  The open $i$-simplex can be realized in  $\mathbb{R}^{i+1}$ where the coordinates correspond to the lengths of the $i+1$ edges.  In this way, moving within a simplex changes only the metric on $\G$.

Often the set of graphs is restricted further to exclude graphs with separating edges.  We will employ this restriction and denote the resulting space as $\A_n$.  After restricting to graphs without separating edges, the geometric realization of the poset of marked graphs is the simplicial spine $\SA_n$ which is contained in the barycentric subdivision of $\A_n$.   Auter space deformation retracts to $\SA_n$.  The deformation retraction alters only the metrics of the graphs so, for every open simplex $\Delta$ in $\A_n$, points in $\Delta$ remain so throughout the deformation retraction.

For $\alpha \in Aut(F_n)$, there is a homotopy equivalence $\alpha' \colon R_n \to R_n$ that realizes $\alpha$.  This gives a right action of $Aut(F_n)$ on $\A_n$ by precomposition, $(\G, \phi) \alpha = (\G, \phi \circ \alpha')$.  This action only alters the marking on $(\G, \phi)$ and does so by applying $\alpha^{-1}$ the labels on $\G$.

\subsection{Ribbon Graphs}

A \emph{ribbon graph} is a finite connected graph $\G$ with, for each vertex in $\G$, a cyclic ordering of the half edges incident to that vertex.  The cyclic orderings are together called the \emph{ribbon structure} of the ribbon graph and will be denoted $O(v)$ where $v$ is a vertex of $\G$.  We refer to $\G$ as the \emph{underlying graph}, and denote a ribbon graph by the pair $(\G, O)$. 

As in \cite{mulase-penkava}, given a ribbon graph $(\G, O)$ we can construct an orientable basepointed surface.  The construction consists of gluing once punctured disks on to $\G$ so that the boundaries of the disks identify with certain edge cycles.  Let  $E = \{e_{1}, e_{2}, \ldots, e_{l-1}, e_{l} \}$ be a cyclic ordered list of directed edges in $\G$, where $i(e_k)$ and $t(e_k)$ are the initial and terminal vertices of $e_k$.  This list $E$ is a \emph{boundary cycle} for $(\G, O)$ if $t(e_k) = i(e_{k+1})$ and $e_{k+1}$ immediately follows $e_k$ in the order $O(t(e_k))$ for all $k$ including $k = l$.  The surface given by gluing once punctured disks onto $\G$ along boundary cycles is the \emph{ribbon surface} of $\G$ and is denoted $|\G, O|$.  Note that $\G$ naturally includes in its ribbon surface.

\begin{remark}
Another way to construct $|\G, O|$ is to fatten the edges of $\G$ so that at each vertex the ends of the ribbons glue together in agreement with $O$.  The surface $|\G, O|$ produced by either edge fattening or gluing once punctured disks depends on the cyclic ordering $O$ given to $\G$.  Consider the graph $\G$  in Frame (A) of Figure \ref{fig:ribbon}.  Frames (B) and (C) of Figure \ref{fig:ribbon} give the result of these surface constructions for two different ribbon structures on $\G$.  The surface in Frame (B) is homeomorphic to a sphere with 3 punctures while the surface in Frame (C) is homeomorphic to a once punctured torus.
\end{remark}

\begin{figure}
\subfloat[]{\fbox{\includegraphics[scale = .3975]{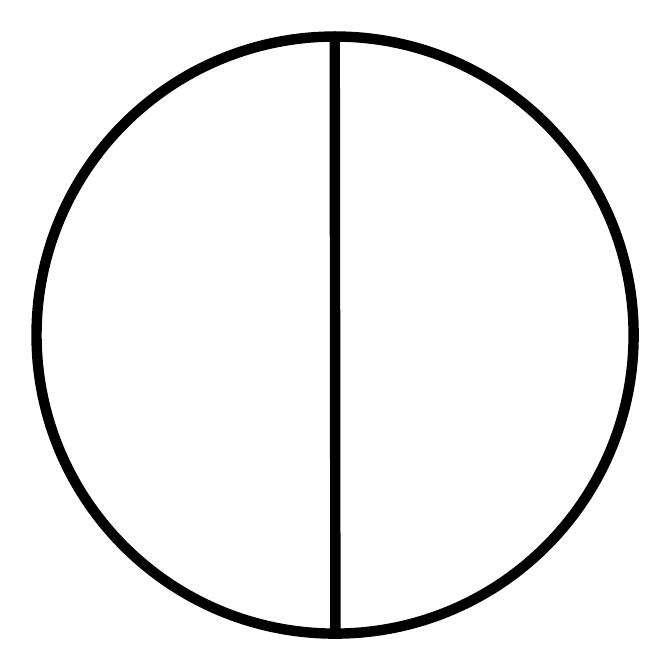}}}
\hfill
\subfloat[]{\fbox{\includegraphics[scale = .35]{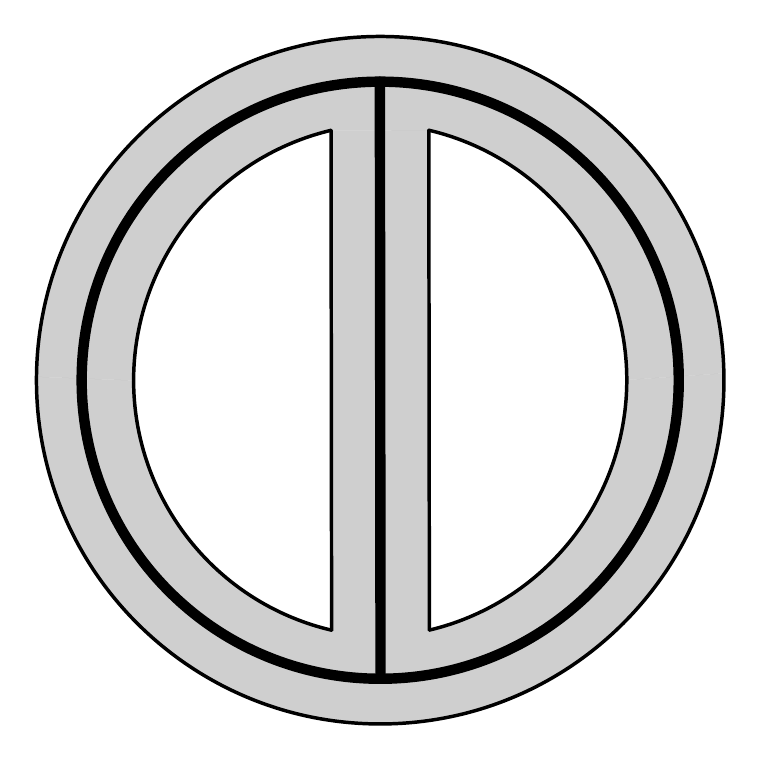}}}
\hfill
\subfloat[]{\fbox{\includegraphics[scale = .358]{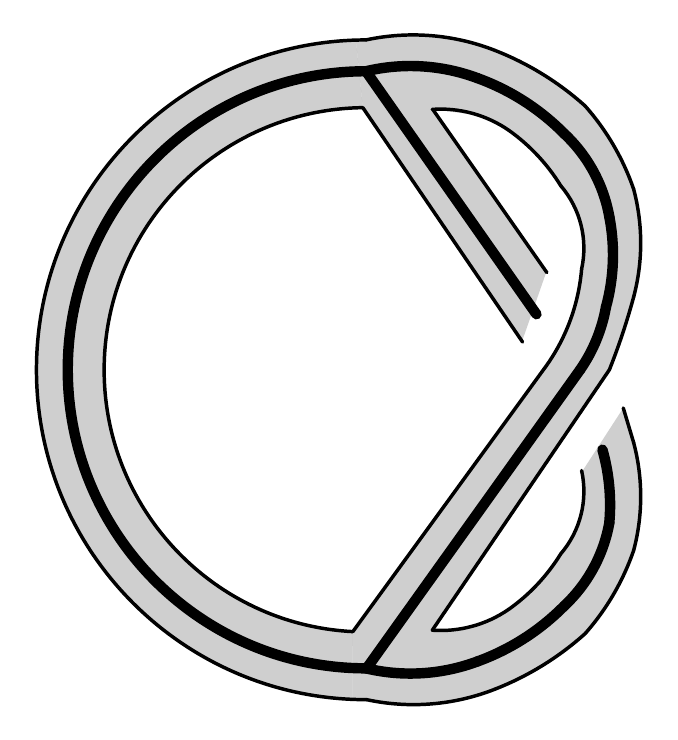}}}
\caption{}
\label{fig:ribbon}
\end{figure}

A \emph{marked surface} is a surface $\Sigma$ together with a homotopy equivalence $\psi \colon R_{2g+p-1} \rightarrow \Sigma$.  Two marked surfaces $(\Sigma_1, \psi_1)$ and $(\Sigma_2, \psi_2)$  are \emph{equivalent} if there exists an orientation preserving homeomorphism $h \colon \Sigma_1 \rightarrow \Sigma_2$ so that $h \circ \psi_1$ is homotopic to $\psi_2$.  A marked graph $(\G, \phi)$ \emph{can be drawn in} the marked surface $(\Sigma, \psi)$ if  there is a ribbon structure $O$ so that $(|\G, O|, i \circ \phi)$ is equivalent to $(\Sigma, \psi)$ where $i \colon \G \rightarrow |\G, O|$ is inclusion.  The ribbon structure $O$ that draws $(\G, \phi)$ in $(\Sigma, \psi)$ is unique.

\begin{definition}
The \emph{ribbon graph space} $\T_{\Sigma}$ for the marked surface $(\Sigma, \psi)$ is the subspace of $\A_{2g+p-1}$ made up of simplicies corresponding to marked graphs that can be drawn in $(\Sigma, \psi)$.  The \emph{ribbon graph complex} $\ST_{\Sigma}$ for the marked surface $(\Sigma, \psi)$ is the subcomplex of $\SA_{2g+p-1}$ spanned by the graphs that can be drawn in $(\Sigma, \psi)$.
\end{definition}

\begin{remark}
While we have defined $\T_{\Sigma}$ and $\ST_{\Sigma}$ as subspaces of $\A_{2g+p-1}$ and $\SA_{2g+p-1}$ respectively, these definitions are more naturally made independent of Auter space.  Our choice to define the ribbon graph space and complex as subspaces in this way comes with an abuse of notation; the marking $\psi$ is not canonical and determines an embedding of $(\T_{\Sigma}, \ST_{\Sigma})$ in $(\A_{2g+p-1}, \SA_{2g+p-1})$.  We omit $\psi$ from the notation as it is not relevant to our arguments.  Further, we use these definitions of $\T_{\Sigma}$ and $\ST_{\Sigma}$ because we will employ the methods of Hatcher and Vogtmann from \cite{hatcher-vogtmann} to prove our main result.  Given a map from a $k$-dimensional sphere $D^k$ to $\A_{2g+p-1}$, Hatcher and Vogtmann constructed an explicit homotopy moving the range of the map into a particular subcomplex of $\A_{2g+p-1}$.  Since $\A_{2g+p-1}$ is contractible, this shows that the subcomplex is $(k-1)$-connected.  The primary observation of this work is that this homotopy can be chosen to preserve $\T_{\Sigma}$.
\end{remark}

The space $\T_{\Sigma}$ is a contractible union of open simplices.  Moreover, if $\Delta$ is an open simplex in $\T_{\Sigma}$ and $\delta$ is one of its faces in $\A_{2g+p-1}$, then $\delta$ is in $\T_{\Sigma}$.  To see this, note that collapsing a forest in a ribbon graph induces a canonical ribbon structure on the resulting graph.  Moreover, let $(\G, O)$ be a ribbon graph, $F$ be a forest in $\G$, and $O'$ be the ribbon structure on $\G / F$ induced by $O$.  Then $|(\G,O)|$ is homeomorphic to $|(\G/F,O')|$ by a map that shrinks the ribbons in $|(\G,O)|$ corresponding to $F$.  Figure \ref{fig:edgecollapse} illustrates the ribbon structure induced by an edge collapse.  Specifically, Frame (B) of Figure \ref{fig:edgecollapse} gives the ribbon graph resulting from collapsing the middle edge in Frame (A).

\begin{figure}
\subfloat[]{\fbox{\includegraphics[scale = .216]{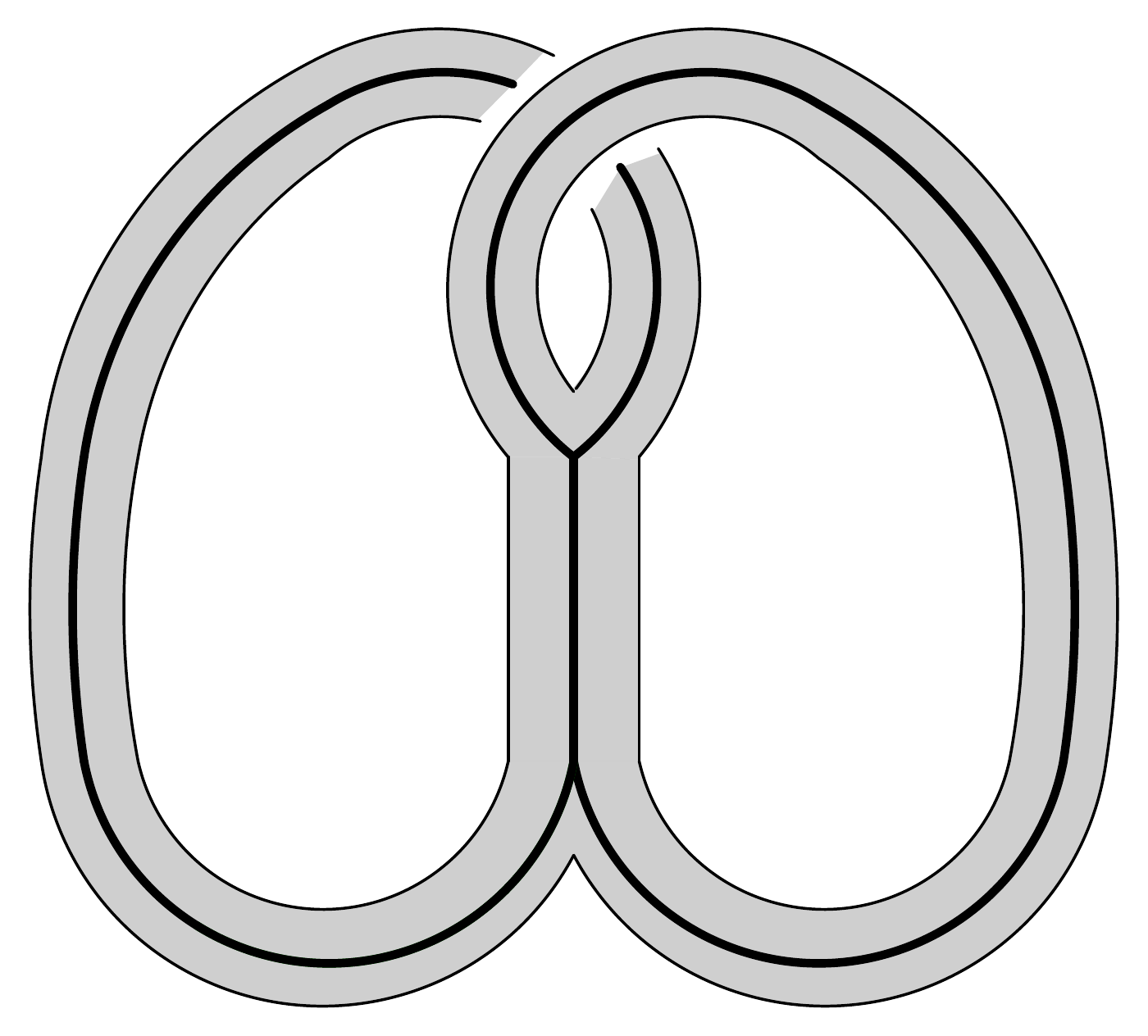}}}
\qquad
\qquad
\subfloat[]{\fbox{\includegraphics[scale = .3]{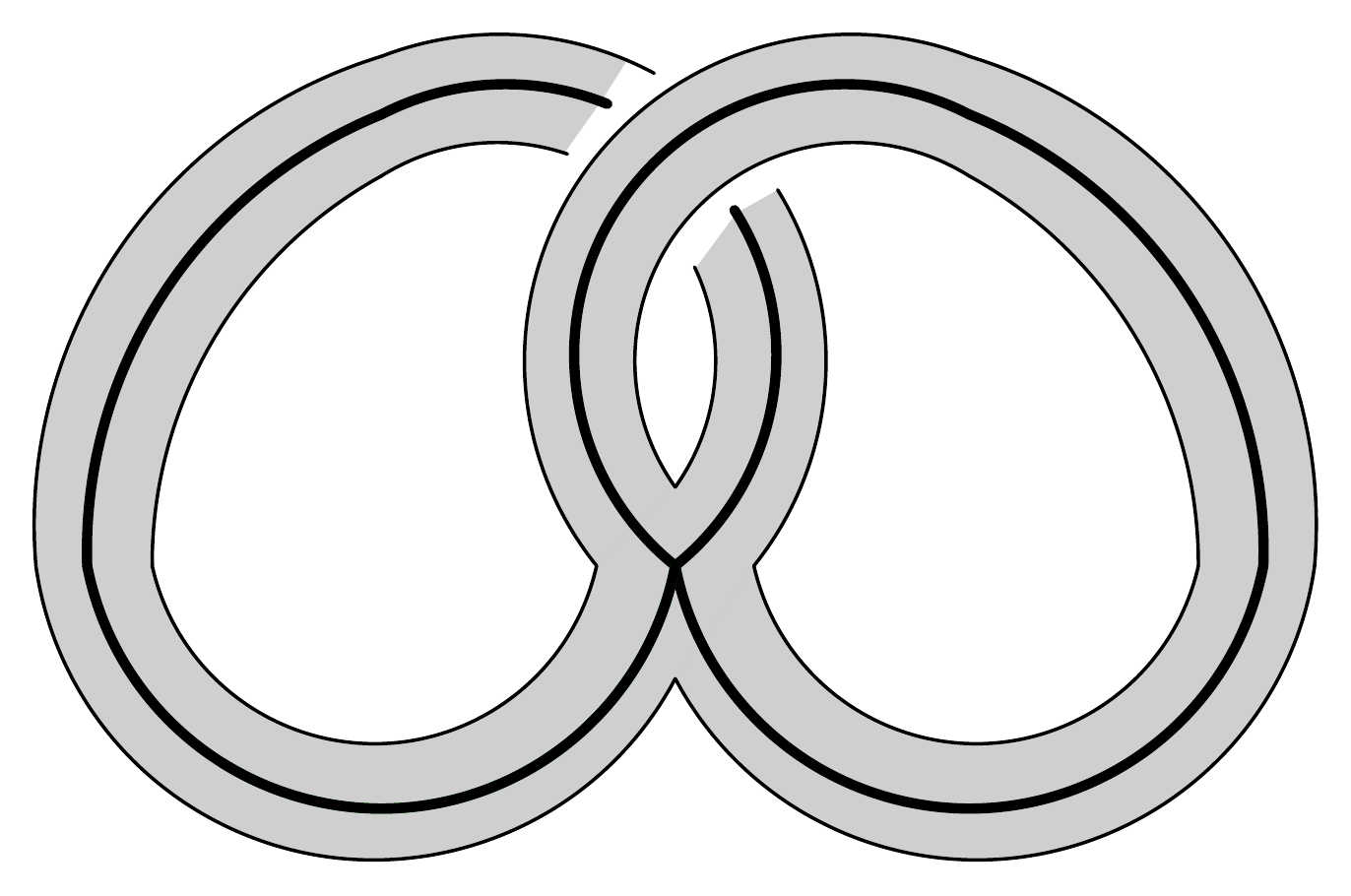}}}
\caption{}
\label{fig:edgecollapse}
\end{figure}

Given $(\G,O)$ with vertex $v$, an \emph{allowed expansion} of $v$ is a partition of the half edges incident to $v$ into two sets of successive half edges.  More formally, the partition $\{ A, B \}$ is an allowed expansion if $O(v)$ can be written as $a_1, a_2, \ldots a_l, b_1, b_2, \ldots b_m$ where $a_i \in A$ and $b_j \in B$ for all $i$ and $j$.  Allowed expansions correspond to edge expansions that can occur within $\T_{\Sigma}$.  Let $(\G,\phi) \in \T_{\Sigma}$ with ribbon structure $O$.  Let $\G'$ be the graph made by replacing vertex $v$ with an edge $e$ and two new vertices $v_1$ and $v_2$ and by attaching the half edges in $A$ to $v_1$ and the half edges in $B$ to $v_2$.  A ribbon structure $O'$ and a marking $\phi'$ are induced by $O$ and $\phi$ respectively so that $(\G', \phi') \in \T_{\Sigma}$ and $O'$ draws $(\G', \phi')$ in $(\Sigma, \psi)$.  Only allowed expansions preserve $\T_{\Sigma}$, graphs resulting from edge expansions that do not respect $O$ cannot be drawn in $(\Sigma, \psi)$.

\begin{figure}
\begin{picture}(400,80)
\put(0,10){\subfloat[]{\fbox{\includegraphics[scale = .4]{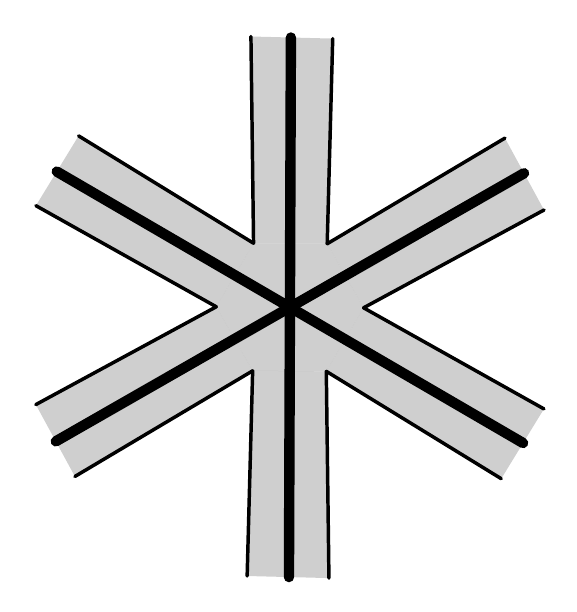}}}}
\put(114,10){\subfloat[]{\fbox{\includegraphics[scale = .324]{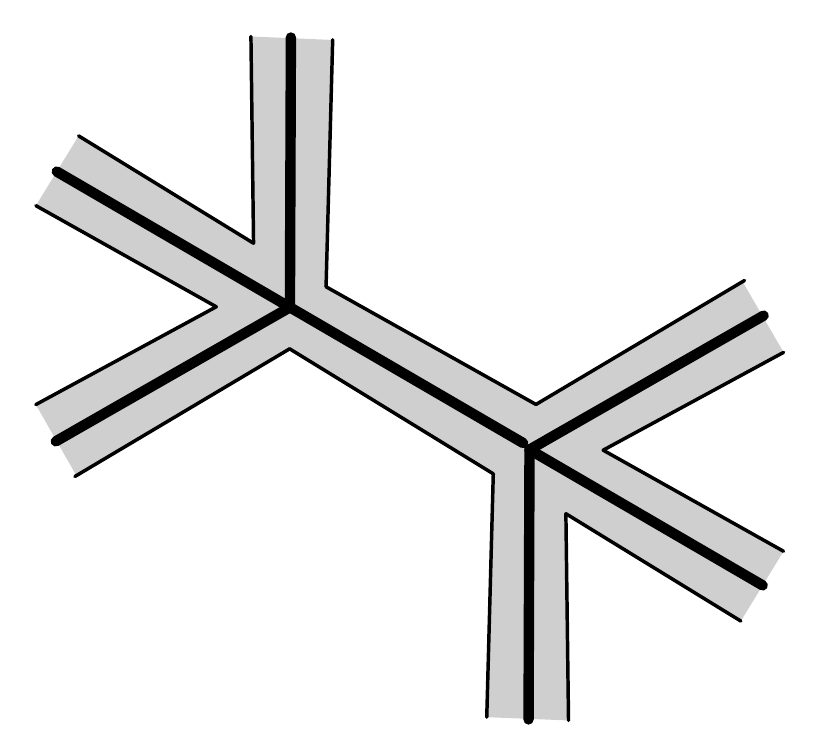}}}}
\put(238,10){\subfloat[]{\fbox{\includegraphics[scale = .314]{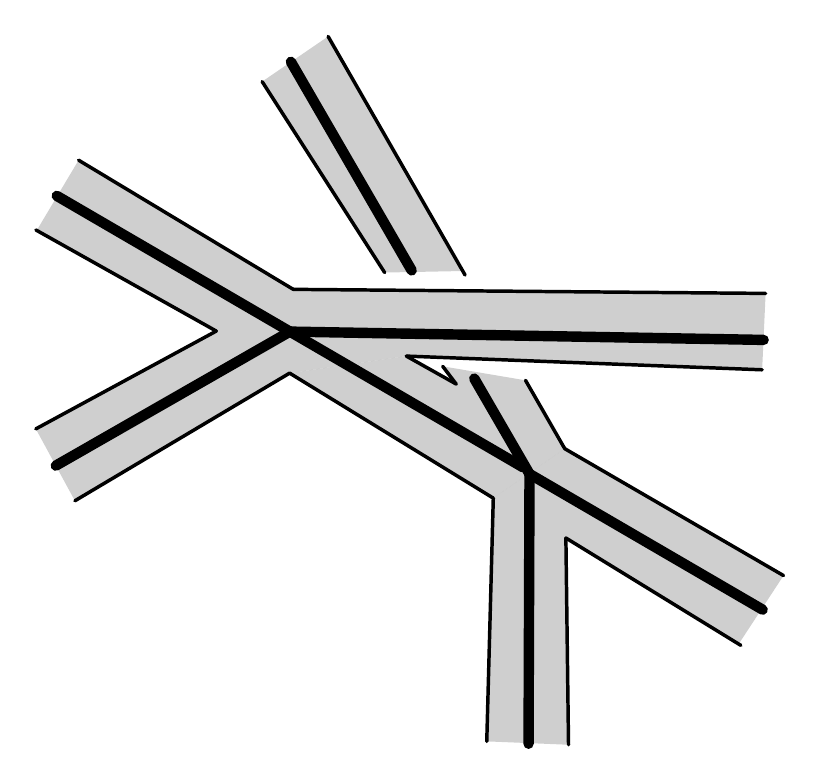}}}}
\put(150,35){$e$}
\put(150,55){$v_1$}
\put(173,20){$v_2$}
\put(273,33){$e$}
\put(263,59){$v_1$}
\put(295,20){$v_2$}
\end{picture}
\caption{}
\label{fig:edgeexpand}
\end{figure}

Frame (B) of Figure \ref{fig:edgeexpand} gives part of a ribbon graph obtained by an allowed expansion of the vertex in Frame (A).  Frame (C) of Figure \ref{fig:edgeexpand} illustrates an edge expansion that is not given by an allowed expansion of Frame (A).  

The homotopy equivalence $\psi \colon R_{2g+p-1} \to \Sigma$ embeds $\m$ as a subgroup into $Aut(F_{2g+p-1})$, where an automorphism is a mapping class if and only if it stabilizes the cycle surrounding each puncture \cite{zieschang}.  The action of $Aut(F_{2g+p-1})$ on $\A_{2g+p-1}$ induces an action of $\m$ on $\A_{2g+p-1}$ and $\T_{\Sigma}$ is preserved by this action.  The action of $\m$ on $\T_{\Sigma}$ only alters the marking on each point, represented by a change in the labels, and leaves the ribbon structure and metric fixed.

\section{\bf Summary of Proof of Degree Theorem}

In this section the basic definitions required for the proof of the degree theorem will be given and the overall idea of the proof will be outlined.  These ideas are introduced in Section 3 of \cite{hatcher-vogtmann}, and are discussed briefly here.

\subsection{Degree Spaces and Subcomplexes}

\begin{definition}
    The \emph{degree} of a finite connected graph $\G$ 
    with basepoint $w$ is:
    \begin{center}
	$\deg(\G) = \displaystyle\sum_{x \in V\G, \; x \not = w} \big( |x| - 
	2 \big)$,
    \end{center}
    where the sum is taken over all non-basepoint vertices and $|x|$ 
    denotes the valence of $x$.
\end{definition}

The degree of a graph generally corresponds to the amount of branching that does not occur at the basepoint.  This definition extends to ribbon graphs; the degree of a ribbon graph is the degree of its underlying graph.  For example, Figure \ref{fig:ribbonsplit} shows a ribbon graph whose bottommost vertex is the basepoint.  This ribbon graph has degree 3.

\begin{figure}
\fbox{\includegraphics[scale = .3]{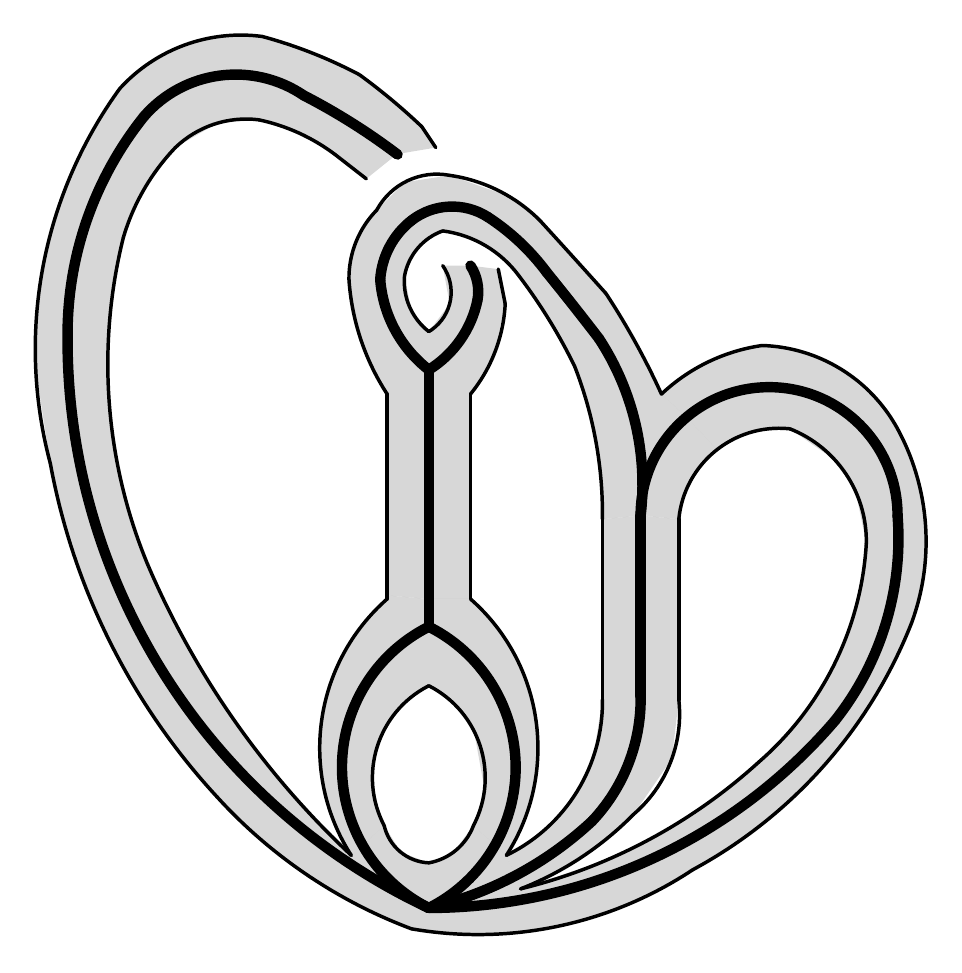}}
\caption{}
\label{fig:ribbonsplit}
\end{figure}

\begin{definition}
    The subspace of $\T_{\Sigma}$ consisting of open simplices corresponding to graphs of degree $k$ or less is the \emph{degree $k$ space} of $\Sigma$ and is denoted $\T_{\Sigma,k}$.  The analogous subspace for $\A_n$ is denoted $\A_{n,k}$. The subcomplex of $\ST_{\Sigma}$ that is spanned by $0$-simplices with underlying graphs of degree $k$ or less is the \emph{degree $k$ complex} of $\Sigma$ and is denoted $\ST_{\Sigma,k}$.  The analogous subcomplex for $\SA_n$ is denoted $\SA_{n,k}$.
\end{definition}

Since the action of $Aut(F_n)$ only alters markings, $Aut(F_n)$ preserves $\A_{n,k}$ and $\SA_{n,k}$.  Similarly, the action of $\m$ preserves $\T_{\Sigma,k}$ and $\ST_{\Sigma,k}$.  The deformation retraction of $\A_n$ to $\SA_n$ only alters the metric on each graph, it leaves the degree of the graphs invariant.  Restricting to $\A_{n,k}$ gives a deformation retraction to $\SA_{n,k}$.  Restricting the deformation retraction to $\T_{\Sigma}$ produces a deformation retraction to $\ST_{\Sigma}$ which can be restricted further to a deformation retraction from $\T_{\Sigma, k}$ to $\ST_{\Sigma, k}$.

Hatcher and Vogtmann's primary technical result in \cite{hatcher-vogtmann} is the following:

\begin{theorem}[Hatcher, Vogtmann]
Auter space contains a sequence of simplicial complexes $\SA_{n, 0} \subset \SA_{n, 1} \subset \ldots \subset \SA_{n, 2n-2} = \SA_{n}$.  The complex $\SA_{n, k}$ is $k$-dimensional, $(k-1)$-connected, and preserved under the action of $Aut(F_n)$.
\label{thm:hv}
\end{theorem}

Hatcher and Vogtmann proved Theorem \ref{thm:hv} by homotoping any piecewise linear map $f \colon D^k \to \A_n$ into $\A_{n,k}$ via a homotopy that monotonically reduces degree which shows that $(\A_n, \A_{n,k})$ is $k$-connected.  Since $\A_n$ is contractible, this proves that $\A_{n,k}$ is $(k-1)$-connected and that the same is true for $\SA_{n,k}$ because $\A_{n,k}$ deformation retracts to $\SA_{n,k}$.  

We will prove the identical result for $\ST_{\Sigma}$ by showing that Hacther and Vogtmann's homotopy can be chosen to preserve $\T_{\Sigma}$.  That is, any piecewise linear map $f: D^k \to \T_{\Sigma}$ can be homotoped into $\T_{\Sigma,k}$ via a homotopy that monotonically reduces degree.  This is sufficient to show that $\T_{\Sigma,k}$ is $(k-1)$-connected because $\T_{\Sigma}$ is contractible.  Recall that the main result of this work is:

\begin{maintheorem}
The space of basepointed ribbon graphs $\T_{\Sigma}$ contains simplicial complexes $\ST_{\Sigma, 0} \subset \ST_{\Sigma, 1} \subset \ldots \subset \ST_{\Sigma, 4g+2p-4} = \ST_{\Sigma}$.  The complex $\ST_{\Sigma, k}$ is $k$-dimensional, $(k-1)$-connected, and preserved under the action of $\m$.
\end{maintheorem}

Most of the statements made in Theorem \ref{thm:main} are straightforward. By definition $\ST_{\Sigma,k} \subseteq \ST_{\Sigma,k+1}$.  Further, $dim(\ST_{\Sigma,k}) = dim(\SA_{2g+p-1,k})$ because any graph $\Gamma$ representing a simplex in $\SA_{2g+p-1,k}$ can be the underlying graph for a ribbon graph that can be drawn in $\Sigma$.  Then Theorem \ref{thm:hv} shows that $\dim(\ST_{\Sigma, k}) = k$ and that $\ST_{\Sigma, 4g+2p-4} = \ST_{\Sigma}$.  We will spend the remainder of this work proving that $\ST_{\Sigma, k}$ is $(k-1)$-connected.

\subsection{Homotopies of Graphs}

The homotopy that we will describe is most simply thought of as a sequence of homotopies on metric graphs.  There are two primary types of graph homotopies that will be used in our proof of Theorem \ref{thm:main}.  They are presented in Subsections 3.2.1 and 3.2.2.

We take inspiration from Morse Theory in defining two key concepts that play important roles in our homotopy: height functions and critical points.

\begin{definition}
For a basepointed metric graph $\Gamma$, the \emph{height function} $h \colon \Gamma \to \mathbb{R}$ gives the distance to the basepoint.
\end{definition}

Taking $h$ to be a literal measurement of height, we can consider downward paths emanating from a point in $\Gamma$.  A point in $\Gamma$ is a \emph{critical point} if there is more than one downward path from the point.  This connects to Morse Theory for manifolds in that the topological type of $h^{-1}([0,r])$ only changes when $r$ passes through a critical value.  This idea is illustrated in Figure \ref{fig:cansplit}.  For each of the graphs, take the bottommost point to be the basepoint.  Instead of using arc length as the distance from each point to the basepoint, consider the height of the point as this distance.  Note that each of the graphs in Figure \ref{fig:cansplit} has four critical points, illustrated by the white dots.  For the remainder of this work, all figures of graphs and ribbon graphs have metrics given by the height function and, when relevant, white dots marking critical points.

\begin{figure}
\subfloat[]{\fbox{\includegraphics[scale = .3]{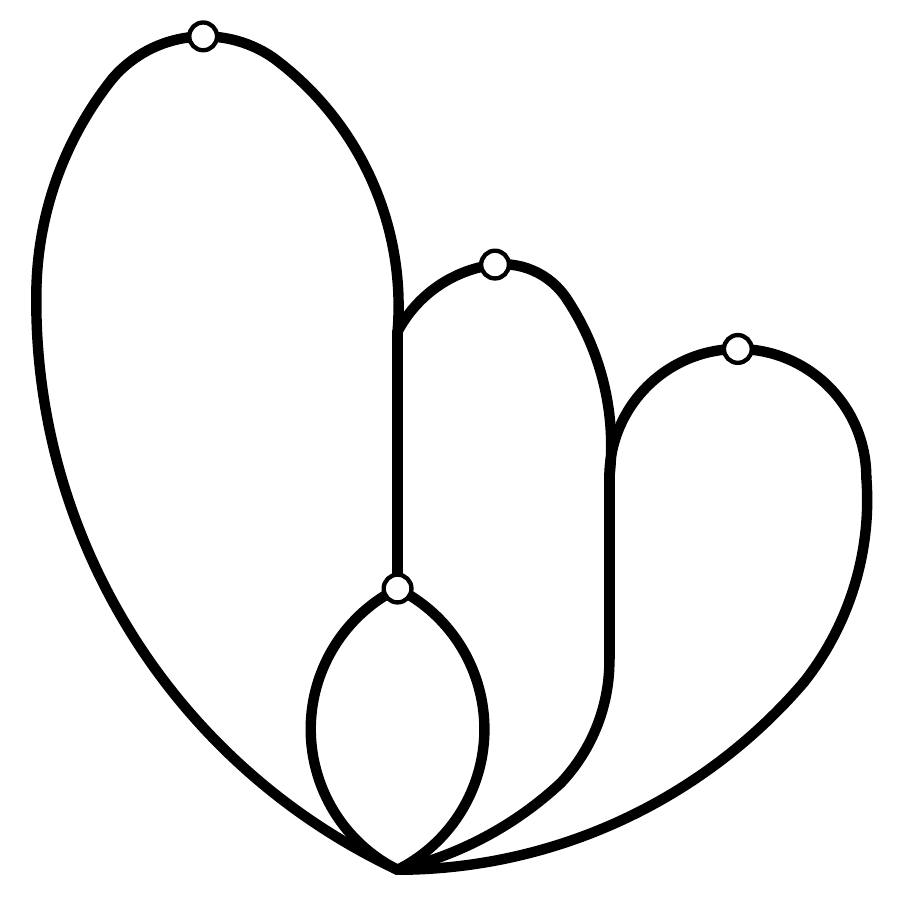}}}
\qquad
\qquad
\subfloat[]{\fbox{\includegraphics[scale = .3]{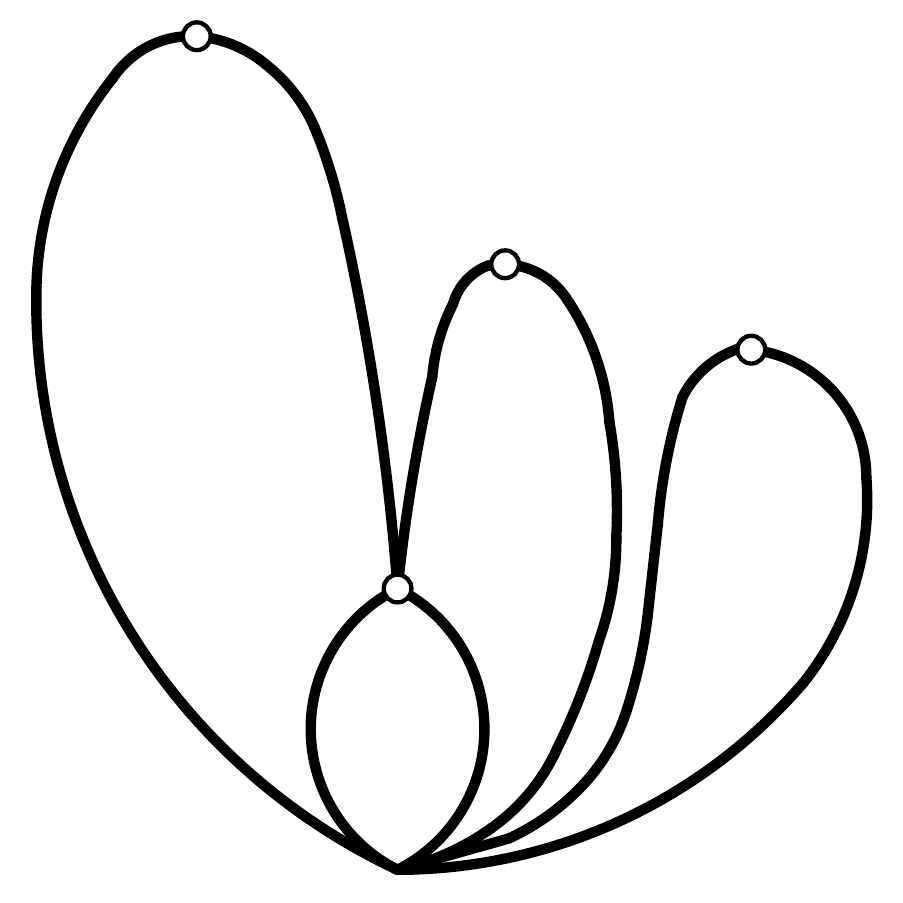}}}
\caption{}
\label{fig:cansplit}
\end{figure}

\subsubsection{Canonical Splitting}

Let $x$ be a critical point.  If $x$ is a vertex the \emph{cone of $x$}, denoted $C_x$, is the union of $x$ and the open downward edges leaving $x$.  If $x$ is not a vertex, $C_x$ is the open edge containing $x$.  A \emph{branch} of a critical point vertex $x$ is the union of $\{ x \}$ and a downward edge from $x$; the cone of $x$ is made up of these branches.  An \emph{extended branch} is a downward path from $x$ that ends at a critical point.

The primary difference between branches and extended branches is that branches only intersect at critical points while two or more extended branches can intersect on edges.  This happens when there are multiple upward edges and only one downward edge leaving a vertex.  The process of contracting, or splitting, all such downward edges is called \emph{canonical splitting}.  This process is canonical; we can start at the top of the graph, working downward, contracting each such edge.  These contractions have the effect of splitting extended branches down to the next critical point, or perhaps all of the way down to the basepoint.  Canonical splitting cannot increase the degree of a graph and decreases degree each time an edge collapse causes an extended branch to reach the basepoint.  An example of canonical splitting is shown in Figure \ref{fig:cansplit}.  The graph in Frame (B) is the result of performing canonical splitting on the graph in Frame (A).  Note that canonical splitting is a forest collapse and so preserves $\T_{\Sigma}$.

\subsubsection{Sliding in $\varepsilon$ cones}

After canonical splitting each branch ends at a critical point.  We can perturb the attaching point of each branch ending at a critical point vertex $x$ downward into the $\varepsilon$ cone $C^{\varepsilon}_x$, the intersection of $C_x$ and the $\varepsilon$ neighborhood of $x$.  We call the homotopy performing such perturbations at each critical point \emph{sliding in $\varepsilon$ cones}.  This process consists of putting the attaching points of extended branches into `general position' relative to the critical point.  Sliding in $\varepsilon$ cones does not alter degree, however following this homotopy with canonical splitting puts the end of each branch at the basepoint, reducing degree.

The first deviation from the work of Hatcher and Vogtmann in \cite{hatcher-vogtmann} occurs in how we handle sliding in $\varepsilon$ cones.  Some care must be taken to insure that this homotopy preserves $\T_{\Sigma}$, we address this issue in Subsection 4.4.  Figure \ref{fig:slidingerror} shows possible results of sliding in $\varepsilon$ cones followed by canonical splitting on ribbon graph (A), whose ribbon surface is a torus with 3 punctures.  The two edges that will be affected by these alterations are labeled $e_1$ and $e_2$.  Specifically, sliding in $\varepsilon$ cones moves these edges off of the central vertex, to either the left or the right.  Graphs (B), (C), (D) and (E) are the results of applying canonical splitting after choosing these directions.  That is, graph (B) corresponds to moving $e_1$ to the right and $e_2$ to the left.  Graph (C) occurs when canonical splitting is applied after moving both edges to the left, while graph (D) is result when both edges are moved to the right.  Lastly, moving $e_1$ to the left and $e_2$ to the right yields graph (E).  However, the directions in which $e_1$ and $e_2$ can be shifted are not independent; while ribbon graphs (B), (C), (D) and (E) are all possible results of applying these homotopies on graph (A), only graphs (B), (C) and (D) have ribbon surfaces homeomorphic to a torus with 3 punctures.  Ribbon graph (E) has a ribbon surface homeomorphic to a sphere with 5 punctures.

\begin{figure}
\begin{picture}(330,175)
\put(0,95){\subfloat[]{\fbox{\includegraphics[scale = .203]{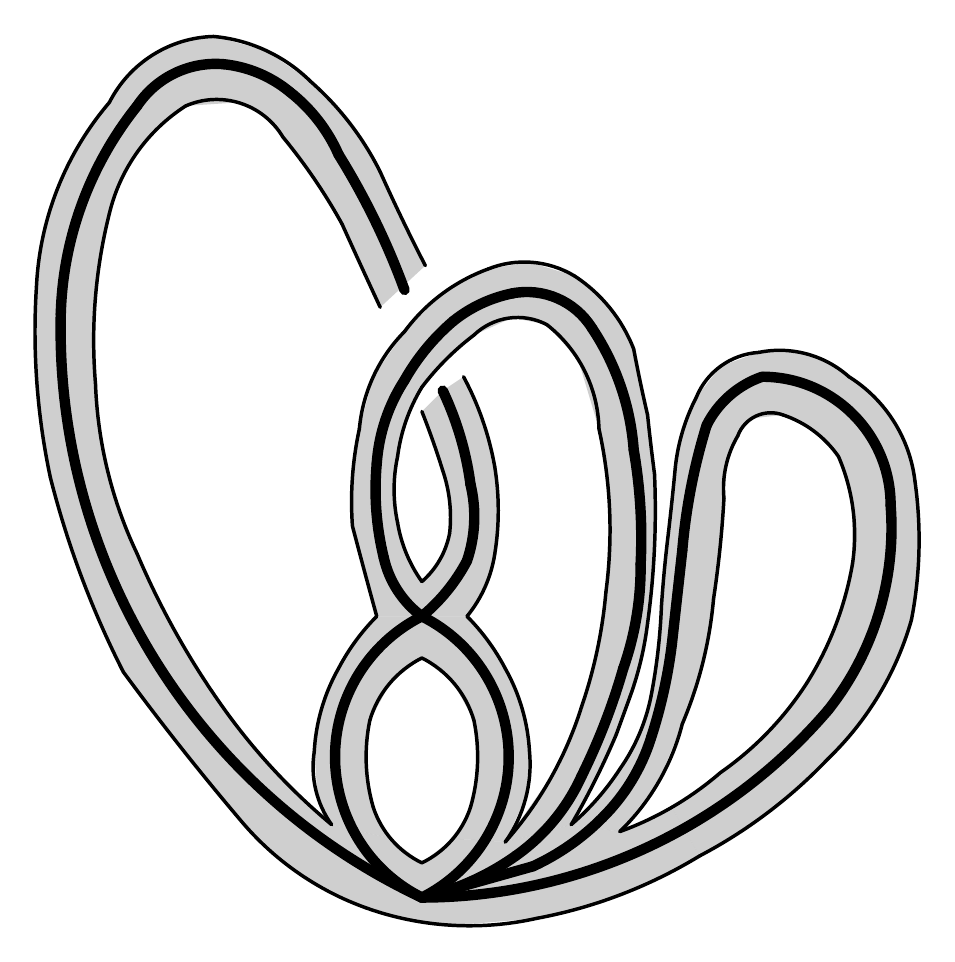}}}}
\put(97,95){\subfloat[]{\fbox{\includegraphics[scale = .25]{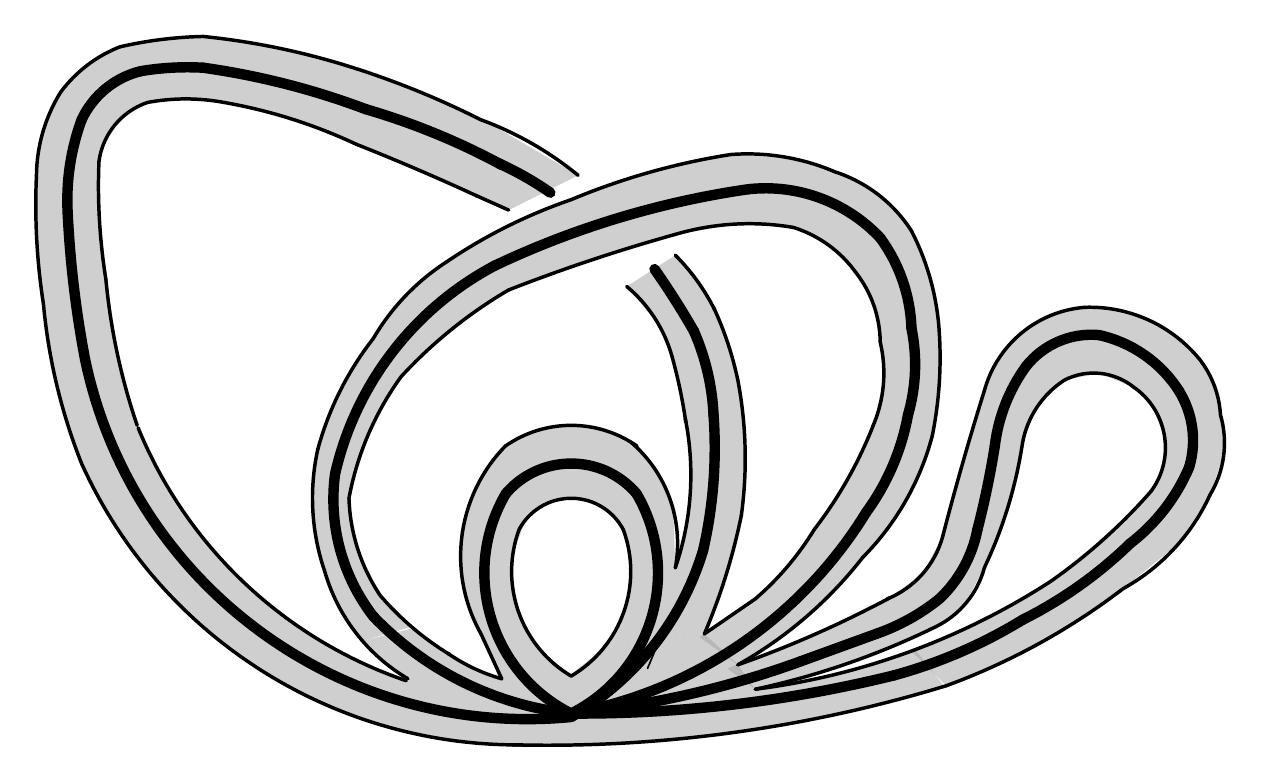}}}}
\put(220,95){\subfloat[]{\fbox{\includegraphics[scale = .249]{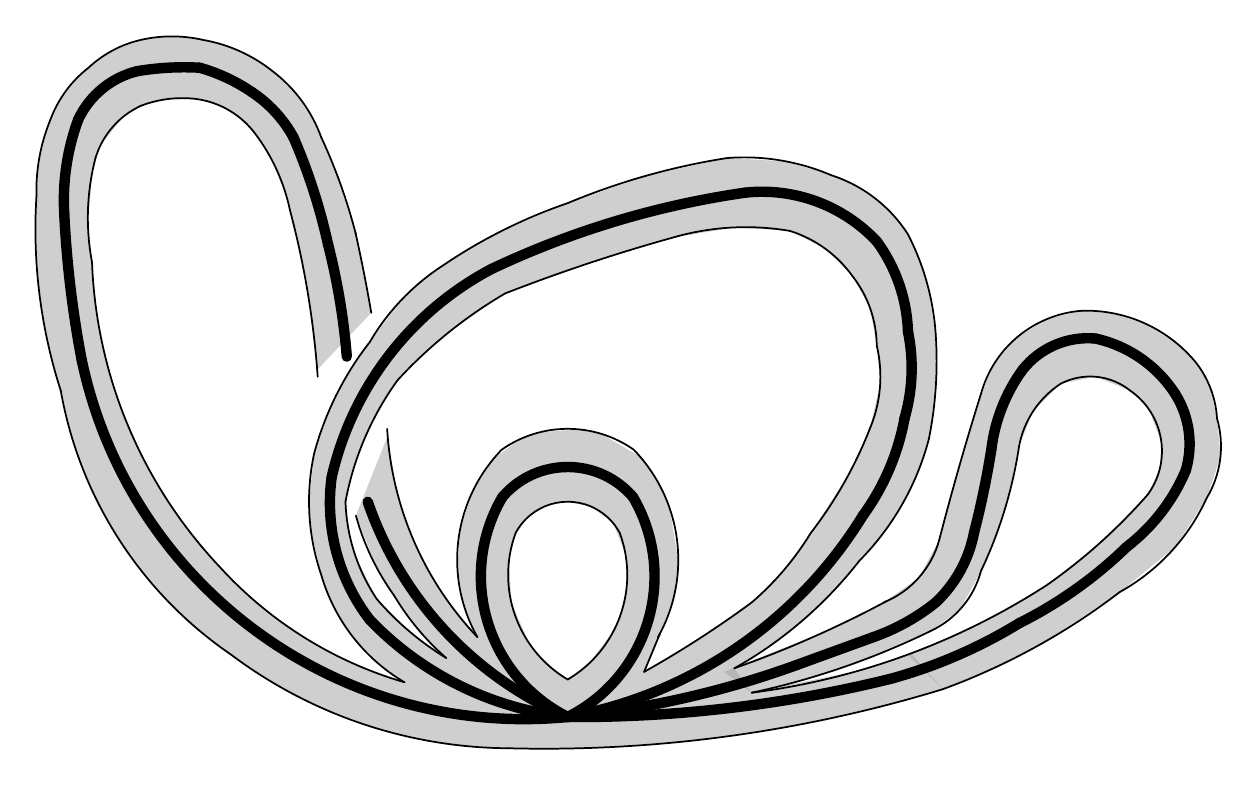}}}}
\put(37,10){\subfloat[]{\fbox{\includegraphics[scale = .25]{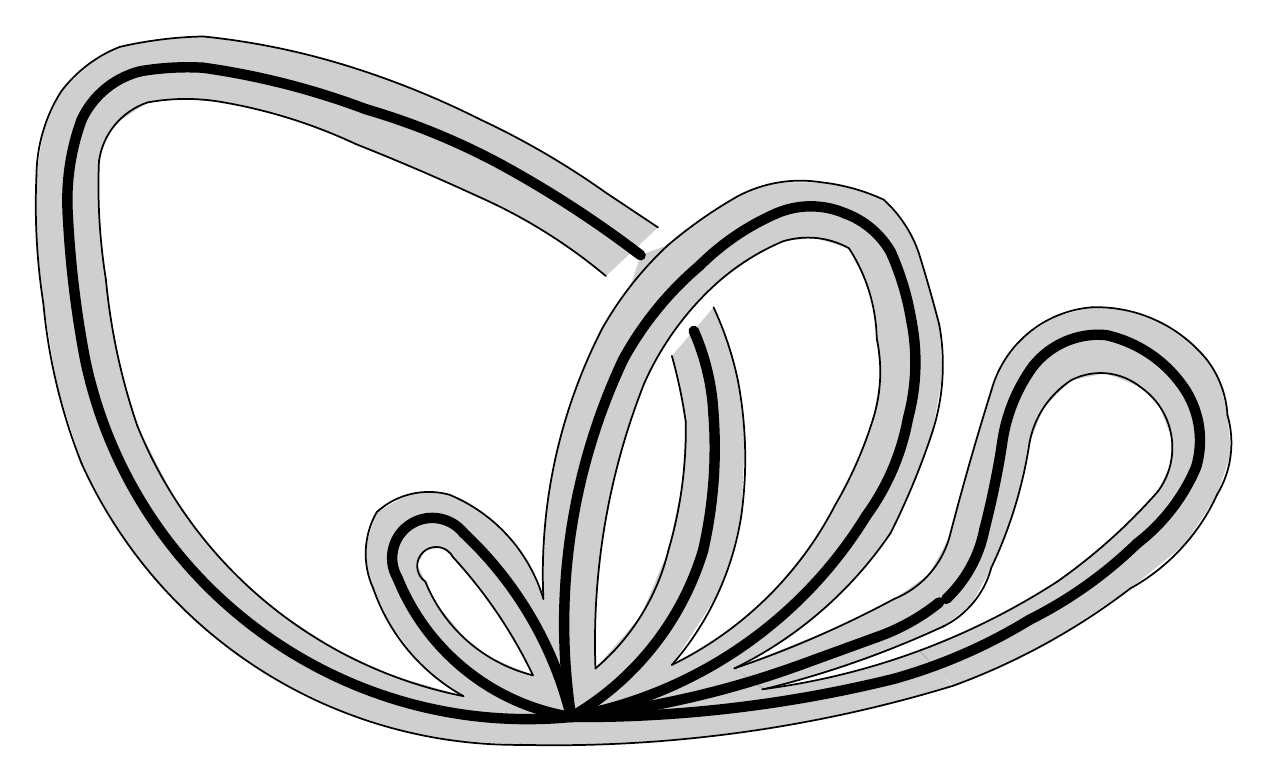}}}}
\put(183,10){\subfloat[]{\fbox{\includegraphics[scale = .249]{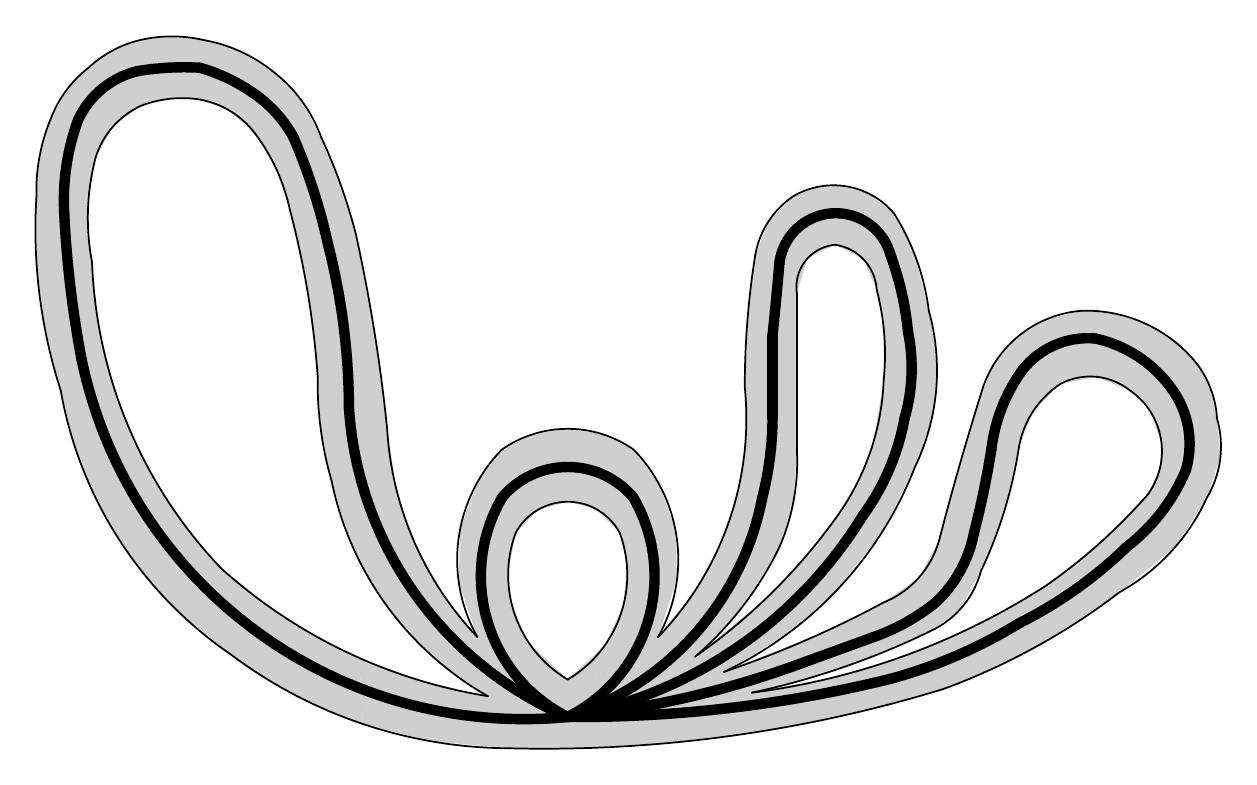}}}}
\put(10,98){$e_1$}
\put(35,139){$e_2$}
\end{picture}
\caption{}
\label{fig:slidingerror}
\end{figure}

\section{\bf Homotopy}

In this section, we discuss the homotopy described by Hatcher and Vogtmann in Section 4 of \cite{hatcher-vogtmann}.   Both canonical splitting and sliding in $\varepsilon$ cones play a significant role in this homotopy.  While each of these are easy to apply to any particular graph, applying them to a parameterized collection of graphs requires consistent $\varepsilon$ cones within the collection.  So, before applying these techniques we perturb the critical points, simplifying their $\varepsilon$ cones.  This perturbation is described in Subsection 4.1.  In Subsection 4.2 we describe the homotopy as an inductive process reducing the complexity of graphs.  Subsection 4.3 details the use of canonical splitting in the homotopy, specializing the discussion in Section 4.5 of \cite{hatcher-vogtmann} to demonstrate that canonical splitting preserves $\T_{\Sigma}$.  The primary deviation from the work of Hatcher and Vogtmann occurs in Subsection 4.4 in which we apply sliding in $\varepsilon$ cones.  Specifically, Hatcher and Vogtmann's homotopy need not preserve $\T_{\Sigma}$ and some extra work is needed to insure that our homotopy does.

\subsection{Perturbation to General Position}

We aim to perturb graphs to simplify critical points by reducing the number of edges in their $\varepsilon$ cones:

\begin{definition}
The \emph{codimension of a critical point} $x$ is defined to be 0 if $x$ is in the interior of an edge and one less than the number of downward edges from $x$ if $x$ is a vertex.  The \emph{codimension of a ribbon graph} is the sum of the codimensions of its critical points.
\end{definition}

Recall that each $i$ simplex in $ \T_{\Sigma}$ can be constructed by varying the edge lengths of its corresponding ribbon graph $\Gamma$.  That is, the $i$ simplex can be thought of as the subspace $\mathbb{R}^{i+1}$, where the coordinates correspond to the lengths of the $i+1$ edges of $\Gamma$, in which each coordinate is positive and the sum of the coordinates is 1.  The metric ribbon graphs of codimension $j$ define a codimension $j$ linear subspace of the open simplex.  To see this, suppose critical point $x$ is incident to $b$ downward edges and hence has codimension $b-1$.  These $b$ edges are the initial edges of $b$ distinct downward paths from $x$ to the basepoint.  Since each of these paths must have the same length, they define $b-1$ linear equations of the edge lengths.  Each of the paths contains at least one edge not contained by any other path, making these equations linearly independent, and giving a solution set with codimension $b-1$.

To begin our homotopy, we will refer to our initial map as $f \colon D^k \to~\T_{\Sigma}$.  Recall that $f$ is piecewise linear and our goal is to homotope $f$ into $\T_{\Sigma,k}$ via a homotopy that monotonically reduces degree.  Our first step is to perturb this map, putting the $\varepsilon$ cones of graphs in the range of the map into the simplest form possible by reducing their codimension. Let $\mathcal{S}_i \subseteq~\T_{\Sigma}$ be the set of points of codimension at least $i$.  This gives a filtration $\T_{\Sigma}~=~\mathcal{S}_0~\supseteq~\mathcal{S}_1~\supseteq~\mathcal{S}_2~\ldots $.  Our goal in this initial stage of the homotopy is to reduce the dimension of the pre-image of $\mathcal{S}_i$ as this simplifies the $\varepsilon$ cones of graphs in the image of our map.

\begin{proposition}
The map $f \colon D^k \to \T_{\Sigma}$ can be homotoped preserving degree to a piecewise linear map so that for all $i$  the codimension of $f^{-1}(\mathcal{S}_i)$ is at least $i$ in $D^k$.
\label{prop:codim}
\end{proposition}

\begin{proof}
Hatcher and Vogtmann showed this for maps $f \colon D^k \to \A_{2g+p-1}$ in Lemma 4.4 of \cite{hatcher-vogtmann}.  Their perturbation does not create nor collapse edges, only alters the lengths of those edges.  Hence, the homotopy preserves simplices and therefore preserves $\T_{\Sigma}$, giving the result of this proposition.
\end{proof}

Proposition \ref{prop:codim} relegates the pre-images of graphs with more complicated $\varepsilon$ cones to the largest codimension achievable by perturbation.  Both canonical splitting and sliding in $\varepsilon$ cones leaving the $\varepsilon$ cones unaltered; the remainder of our homotopy will leave $\varepsilon$ cones fixed, and hence will not disturb the general position established in Proposition \ref{prop:codim}.

\subsection{Idea of Homotopy: Inductive Complexity Reduction}

We will apply canonical splitting and sliding in $\varepsilon$-cones as described in Section 3 to reduce the degree of the points in the image of $f$.  In order to apply these techniques, we take a piecewise approach, homotoping pieces of $D^k$ whose images have consistent $\varepsilon$ cones.  Let $\mathscr{S}_i = f^{-1}(\mathcal{S}_{k-i})$.  Hatcher and Vogtmann showed in Sections 4.2 and 4.3 of \cite{hatcher-vogtmann} that the filtration of $\A_{n}$ analogous to our sets $\mathcal{S}_i$ pulls back to a filtration of $D^k$ by subpolyhedra, and their proof holds in our setting.  Specifically, $D^k = \mathscr{S}_k \supseteq \mathscr{S}_{k-1} \supseteq \mathscr{S}_{k-2} \supseteq \ldots$ is a filtration by subpolyhedra and $\dim(\mathscr{S}_i) \leq i$ by Proposition \ref{prop:codim}.  Most importantly, consider $S$, a connected component of $\mathscr{S}_i - \mathscr{S}_{i-1}$.  Observe that $\varepsilon$ cones of $f(s)$ are constant as $s$ varies in $S$.  To make $S$ a compact polyhedron, delete from $S$ points in a neighborhood of $\mathscr{S}_{i-1}$.

For simplicity, we will relabel $f \colon D^k \to \T_{\Sigma}$ as $f(s) = \G_s$.  This is an abuse of notation as elements of $\T_{\Sigma}$ are marked basepointed graphs.  However, it is straightforward to alter the basepoint and marking throughout the homotopy, so our notation is sufficient.  Since $\G_s \in \T_{\Sigma}$, recall that there is a unique ribbon structure $O_s$ that draws $\G_s$ in $\Sigma$.  Demonstrating that our homotopy preserves $\T_{\Sigma}$ is a matter of showing that we can define $O_{st}$ that draws $\G_{st}$ in $\Sigma$ for all $s$ and $t$.

We will proceed by induction on $i$, homotoping $\Gamma_s$ to degree at most $k$ for all $s$ in a neighborhood of $\mathscr{S}_i$.  Our induction step will decrease the \emph{complexity} $c_s$ of $\Gamma_s$, the number of downward paths in $\Gamma_s$ that begin and end at (and possibly contain) critical points.  We will also be interested in the number $e_s$ of such paths that are extended branches, i.e. that contain no critical points in their interior.  Specifically, if there exists $s \in S$ with $e_s > i$, we will describe a homotopy supported on a neighborhood of $S$ that (1) does not increase degree nor alter the $\varepsilon$ cones of $\Gamma_s$ and (2) lowers the maximum complexity $c_s$ over $S$.  This process guarantees that we can homotope $\Gamma_s$ to achieve $e_s \leq i$ for all $s \in S$ as complexity can only be reduced finitely many times.  Then, canonical splitting leaves $\Gamma_s$ with degree at most $k$, completing the induction step.  While we will describe the homotopy for $S$, a connected component of $\mathscr{S}_i - \mathscr{S}_{i-1}$, the same procedure reduces the degree of $\G_s$ for $s$ in a neighborhood of $\mathscr{S}_j$ where $j$ is minimal.  Hence, we will have also described our base case.

\subsection{Preparatory Canonical Splitting}

In order to perform canonical splitting on $\Gamma_s$ both the $\varepsilon$ cones and the set of extended branches to be split must remain constant as $s$ varies.  That is, complexity must be constant.  Let $K$ be a connected component of the set of $s \in S$ with $c_s$ maximal.  The set $K$ is a closed subpolyhedron of $S$ since an extended branch must move off of a critical point in order to reduce complexity.  By definition, we can perform canonical splitting on $K$, giving a homotopy $\Gamma_{st}$ for all $s \in K$.  Hatcher and Vogtmann extend this homotopy to a neighborhood of $K$ in Section 4.5 of \cite{hatcher-vogtmann}.  Their homotopy meets our needs; below we briefly describe the extension and show that it preserves $\T_{\Sigma}$.  In the intersection of the neighborhood of $K$ with $S$, graphs $\Gamma_s$ may have had the attaching points of some extended branches move off of critical points into $\varepsilon$ cones, lowering complexity.  The homotopy is extended to these graphs by moving the endpoints of the extended branches into the $\varepsilon$ cones as $s$ varies, and leaving the attaching points of those branches fixed as $t$ varies.  For $s$ in a neighborhood of $K$ but outside of $S$ and let $s' \in K$ be a nearby point.  Critical points of $\Gamma_{s'}$ can bifurcate forming several critical points in $\G_s$.  Hatcher and Vogtmann constructed $H_x$ the convex hull of these critical points in $\G_s$ and $H^{\varepsilon}_x$ the union of $H_x$ and its downward $\varepsilon$ neighborhood.  The homotopy extends to $s$ by moving the attaching points of extended branches within $H^{\varepsilon}_x$ and bifurcating the critical points as $s$ varies, and not splitting any branches that attach within $H^{\varepsilon}_x$ as $t$ varies.  We damp down this homotopy to leave the splitting supported only on a neighborhood of $K$, thus creating a homotopy $\Gamma_{st}$ defined over all $D^k$.

\begin{figure}
\begin{picture}(205,90)
\put(0,10){\subfloat[]{\fbox{\includegraphics[scale = .4]{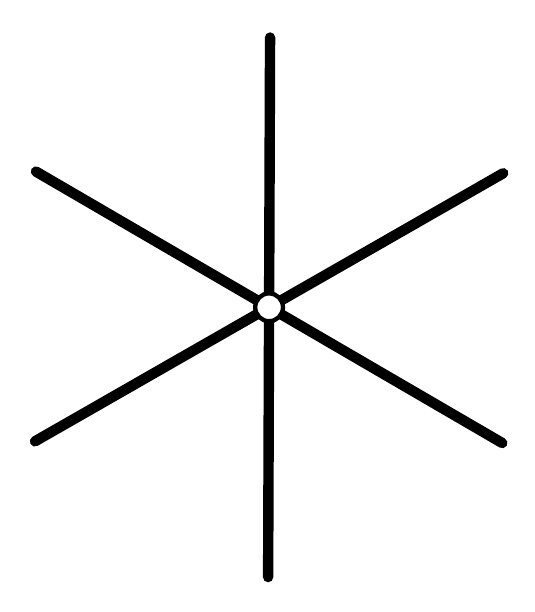}}}}
\put(110,10){\subfloat[]{\fbox{\includegraphics[scale = .354]{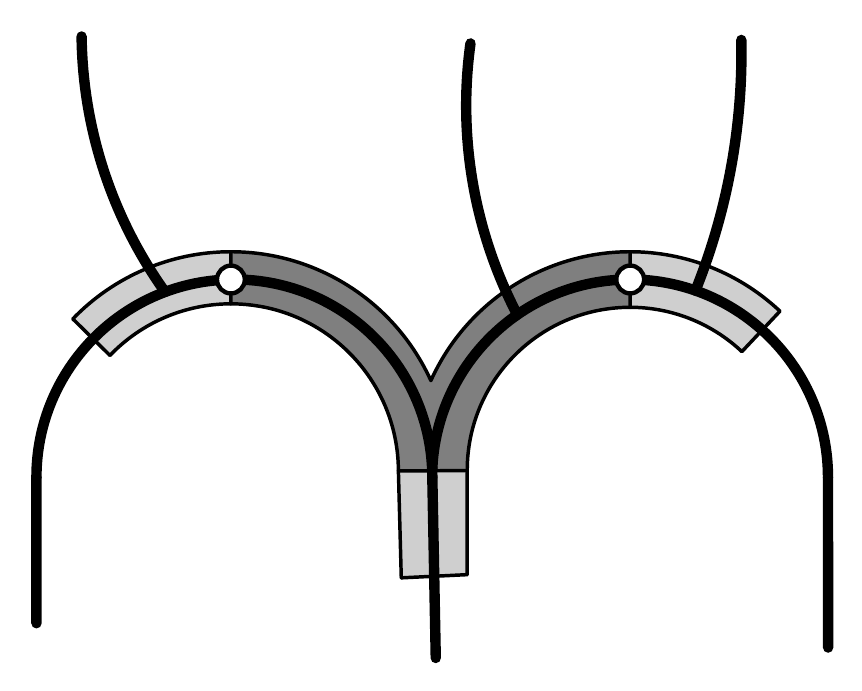}}}}
\put(2,65){1}
\put(2,20){2}
\put(38,10){3}
\put(38,75){4}
\put(60,65){5}
\put(60,20){6}
\put(114,73){1}
\put(119,10){2}
\put(161,10){3}
\put(154,73){4}
\put(193,73){5}
\put(191,10){6}
\end{picture}
\caption{}
\label{fig:attach}
\end{figure}

Figure \ref{fig:attach} shows a possible bifurcation.  Frame (A) shows a critical point vertex $x$ that is incident to six half edges in $\G_{s'}$ with $s' \in K$.  The numbering gives the cyclic order $O_{s'}(x)$.  Frame (B) shows that same section of $\G_s$ for $s$ in a neighborhood of $K$ but outside of $S$.  The dots are the critical points in the convex hull $H_x$, the dark shaded region.  The union of $H_x$ and the region with lighter shading is $H_x^{\varepsilon}$.  

Since canonical splitting is a forest collapse, it preserves $\T_{\Sigma}$, so for all $s \in K$ we have that $\Gamma_{st}$ is a point in $\T_{\Sigma}$.  For $s$ in a neighborhood of $K$, both inside and outside of $S$, note that moving the attaching points of extended branches off of a critical point $x$ and into its $\varepsilon$ cone is an edge expansion.  Additionally, bifurcating critical points requires edge expansions.  In order for $\Gamma_s \in \T_{\Sigma}$, these edge expansions must respect the cyclic order at $x$ as described in Section 2.2.  For example in Figure \ref{fig:attach}, the numbering in Frame (B) gives the correspondence between half edges in Frame (A) and those leaving $H_x^{\varepsilon}$ in Frame (B).  Moreover, this numbering demonstrates that only allowed expansions were used in bifurcating $x$ and induces a ribbon structure $O_s$ that draws $\G_s$ in $\Sigma$.  We can choose $\varepsilon$ small enough so that none of the branches that are being split attach in $C^{\varepsilon}_x$ or $H^{\varepsilon}_x$ in the cases that $s \in S$ and $s \not \in S$ respectively.  No split branch passes a fixed branch during the homotopy and $\Gamma_{st} \in \T_{\Sigma}$ for all $t$.  Hence, the induced ribbon structure $O_{st}$ draws $\G_{st}$ in $\Sigma$ for all $s$ in a neighborhood of $K$ and for all $t$.

\subsection{Complexity Reduction}

We will now reduce the maximum complexity $c_s$ over $S$ by sliding in $\varepsilon$ cones.  At this point, it is necessary to add to the arguments of Hatcher and Vogtmann in \cite{hatcher-vogtmann}, as the homotopy they described does not necessarily preserve $\T_{\Sigma}$.  After performing canonical splitting as described in the previous subsection, for every $\Gamma_s$ with $s$ in a neighborhood $N$ of $K$ in $S$,  the attaching point $\alpha_j$ of every branch $\beta_j$ is either at the basepoint or in an $\varepsilon$ cone of a critical point.  We will construct a homotopy that perturbs the attaching points $\alpha_j$ that lie in $\varepsilon$ cones.  Consider each attaching point $\alpha_j$ as a map $\alpha_j \colon N \to C^{\varepsilon}_x$ where $s \mapsto x$ for all $s \in K$, and perturbations of $\Gamma_s$ can be viewed as perturbations of the maps $\alpha_j$.  To guarantee that $\Gamma_s$ remains in $\T_{\Sigma}$ throughout the perturbation, note that the image $\alpha_j$ must be contained in two specific downward edges, the downward edges before and after $\beta_j$ in the cyclic order of half edges at $x$.  Sliding $\alpha_j$ in $C^{\varepsilon}_x$ can be viewed as expanding vertex $x$ into an edge $e$ and two vertices $x_1$ and $x_2$, in which exactly one of the downward edges incident to $x$ is now incident to $x_1$ while all others are incident to $x_2$.  As described in Section 2.2, to guarantee that the resulting graph is a point in $\T_{\Sigma}$ the edges incident to $x_1$ and $x_2$ must be chosen in agreement with the cyclic order of the original graph.  We will spend the remainder of this subsection working around this technical issue.

Let $e_1, e_2, \ldots, e_p$ be the half edges incident to $x$ written in cyclic order. For an upward half edge $e_j$, the \emph{positive downward direction} of $e_j$ is the first downward half edge encountered by going forward in the cyclic order of $x$ from $e_j$.  Similarly, the \emph{negative downward direction} of $e_j$ is the first downward half edge encountered by going backward in the cyclic order from $e_j$.  The \emph{attaching interval} of $e_j$ is the intersection of $C^{\varepsilon}_x$ with the positive and negative downward directions of $e_j$; this is the set of possible values $\alpha_j(s)$ after sliding in $C^{\varepsilon}_x$.  Note that the attaching interval of an upward half edge is isometric to the interval of the real line $(-\varepsilon, \varepsilon)$, and we make this identification.

While the value $\alpha_j(s)$ can be any point in the attaching interval of $\beta_j$, this set of possible images is contingent on the attaching points of some of the other branches.  More specifically, an \emph{attaching set} is a maximal consecutive sublist of upward half edges incident to $x$.   Note that every half edge in an attaching set has the same positive and negative downward directions.  The set of possible values $\alpha_j(s)$ depends on the images of the attaching points of the branches in the same attaching set as $\beta_j$.

Let $e_{a_1}, e_{a_2}, \ldots, e_{a_r}$ be an attaching set written in cyclic order, and consider the direct product $\Pi_1^r (-\varepsilon, \varepsilon)$ of the attaching intervals of these half edges.  Let $y_k$ be the $k$-th coordinate of this direct product.  The \emph{attaching space} of the attaching set is the subset of $\Pi_1^r (-\varepsilon, \varepsilon)$ defined by $y_k \leq y_{k+1}$ for all $k$.  The direct product of the attaching spaces of the attaching sets of $x$ is the attaching space of $x$.  The direct product of the attaching spaces of the critical point vertices of $\G_s$ is the attaching space of $\Gamma_s$, denote this set $A_s$.  For example, the closure of the attaching space of the critical point shown in Frame (A) of Figure \ref{fig:attach} is a triangular prism.

Note that $A_s$ is constant as $s$ varies over $K$, so we drop $s$ from the notation.  By taking the closure of $A$ we make $A$ a compact polyhedron.  Further, these definitions can be extended for points $s \in N$ by considering the attaching space of the graph given by collapsing all edges contained in $C^{\varepsilon}_x$.  Indeed, the attaching space $A$ is well defined for all $s \in N$.  Moving $s$ in $N$ moves the attaching points within $A$.  

The attaching space $A$ is the set of images of possible attaching points $\alpha_j$ that guarantee that $\Gamma_s$ remains in $\T_{\Sigma}$.  That is, for $s \in N$ if the image of the attaching points of $\G_{st}$ is in $A$, then $\G_{st}$ can be constructed from $\G_{s0}$ by edge collapses and allowed expansions, and the ribbon structure $O_{st}$ that draws $\G_{st}$ in $\Sigma$ is induced by $O_{s0}$.

Denote product of the maps $\alpha_j  \colon N \to C^{\varepsilon}_x$ as $\alpha \colon N \to A$. For every $s \in K$, the image $\alpha(s)$ is $(0,0, \ldots, 0)$, as $s \in K$ was canonically split in the previous step of the homotopy.  Further, $\alpha(s) \not = (0,0, \ldots, 0)$ for all $s \in N - K$ as complexity decreases when leaving $K$ and entering $N$.  Note that reducing the complexity of $\Gamma_s$ for all $s \in K$ means making $\alpha(K) \cap (0,0, \ldots, 0) = \emptyset$.  

There is one copy of $[-\varepsilon, \varepsilon]$ in $A$ for each extended branch,  $\dim (A) = e_s$ for each $s \in K$.  Recall that $\dim(N) \leq i$ and hence by general position theory for polyhedra, if $e_s > i$, then $\alpha(N)$ can be perturbed to be disjoint from $(0,0, \ldots, 0)$, reducing the complexity of $\Gamma_s$.  That is, $\Gamma_s$ can be perturbed so that for every $s \in K$ there is an extended branch $\beta_j$ whose attaching point has moved off of a critical point.

\begin{figure}
\begin{picture}(99,80)
\put(0,0){\fbox{\includegraphics[scale = .35]{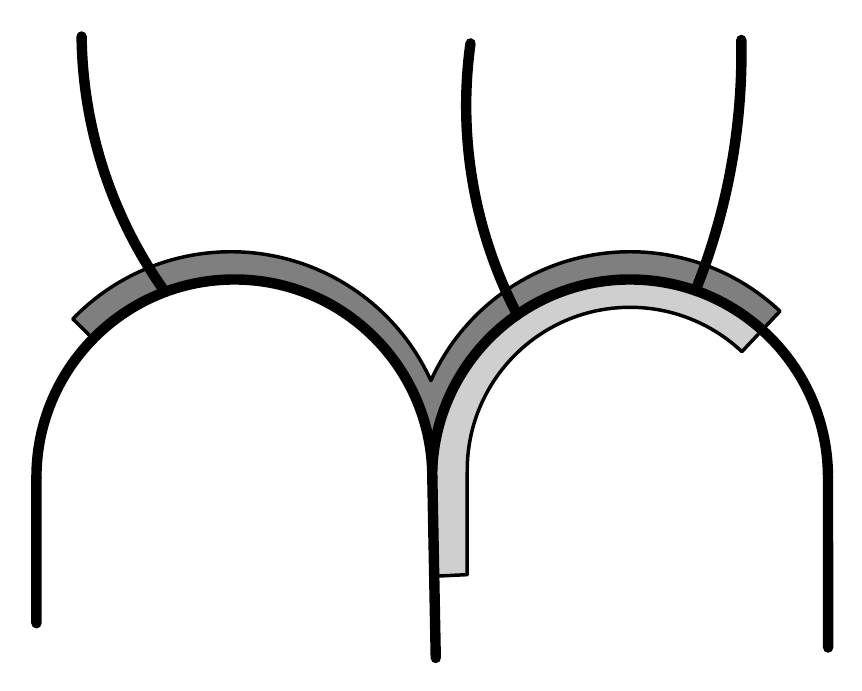}}}
\put(4,63){1}
\put(10,2){2}
\put(51,2){3}
\put(43,63){4}
\put(82,63){5}
\put(80,2){6}
\end{picture}
\caption{}
\label{fig:rcone}
\end{figure}

We can extend the homotopy to points $s$ in a neighborhood of $K$ in $D^k \setminus S$.  Recall that for such $\Gamma_s$, critical points may have bifurcated.  Let $\Gamma_{s'}$ be the graph of a nearby point $s' \in K$.  The endpoints of $C^{\varepsilon}_x$ in $\Gamma_{s'}$ correspond to the endpoints of $H^{\varepsilon}_x$ in $\Gamma_s$.  There is a canonical map from each attaching interval of an extended branch in $\Gamma_{s'}$ to $H^{\varepsilon}_x$ that preserves the endpoint correspondence.  That is, $[-\varepsilon, \varepsilon]$ maps to the geodesic in $H^{\varepsilon}_x$ connecting the corresponding endpoints of $H^{\varepsilon}_x$. Let $r \colon A \to \Pi H^{\varepsilon}_x$ be the product of these maps, where there are $e_{s'}$ factors of $H^{\varepsilon}_x$, one for each upward edge incident to a critical point in $\Gamma_{s'}$. Composing $r$ with the homotopy $\Gamma_{s't}$ extends the homotopy to a full neighborhood of $K$ in $D^k$.  Figure \ref{fig:rcone} shows the image of $r$ in $H_x^{\varepsilon}$ given in Frame (B) of Figure \ref{fig:attach}.  Specifically, the dark shaded region in Figure \ref{fig:rcone} is the image of the attaching space of branch 1 while the light shaded region is the image of the attaching spaces of branches 4 and 5.

Note that $r(\alpha(s'))$ may not be the initial attaching points of upward branches in $\Gamma_s$, and so we precede this homotopy by moving these attaching points to match $r(\alpha(s'))$.  This preceding homotopy is constant on $\Gamma_s$ for all $s \in S$.  

It remains to show that $\G_{st} \in \T_{\Sigma}$ for all $t$.  Given $\G_s \in  \T_{\Sigma}$, the bifurcation of the critical point $x$ can be described in terms of edge expansions.  More specifically, the image under $r$ of the attaching interval of $\beta_j$ is precisely the set of points in $H^{\varepsilon}_x$ that $\beta_j$ can attach at if $\Gamma_s \in \T_{\Sigma}$ and $s$ is in a neighborhood of $K$.   That is, $O_{s't}$ induces a ribbon structure $O_{st}$ that draws $\G_{st}$ in $\Sigma$ for all graphs with attaching points in $r(A)$.  For all $s$ in a neighborhood of $K$, the homotopy $\Gamma_{st}$ preserves $\T_{\Sigma}$.  

Damping down this homotopy outside of $K$, we produce a homotopy that is supported only on a neighborhood of $K$ and that reduces the maximum complexity of $\Gamma_s$ over $s \in N$.  We can repeat this process for each maximum complexity connected component of $S$ until $e_s \leq i$ for all $s \in S$.  

It remains to reduce the degree of $\Gamma_s$ to $k$ for all $s$ in some neighborhood of $S$.  This occurs exactly as described by Hatcher and Vogtmann in Section 4.7 of \cite{hatcher-vogtmann}, through a final application of canonical splitting.  By the arguments given in Subsection 4.3, canonical splitting preserves $\T_{\Sigma}$.

We have homotoped $f$ so that for all $s$ in a neighborhood of $S$, we have $\deg(\Gamma_s) \leq k$.  We can perform this sequence of homotopies on each connected component of $\mathscr{S}_i - \mathscr{S}_{i-1}$, and hence for all $s$ in a neighborhood of $\mathscr{S}_i$, we have monotonically reduced degree so that $\deg(\Gamma_s) \leq k$.  This completes our induction step, and homotopes the image of $f$ into $T_{\Sigma,k}$.  This completes the proof of Theorem \ref{thm:main}.

\section{\bf The Degree 2 Complex and Open Questions}

Our motivation for Theorem \ref{thm:main} comes from successful applications of Theorem \ref{thm:hv} to compute a presentation for and study the homology of $Aut(F_n)$ \cite{afv} \cite{hatcher-vogtmann}.  Both of these works took advantage of the fact that $S\A_{n,2} / Aut(F_n)$ stabilizes for $n \geq 4$.  In this case, $S\A_{n,2} / Aut(F_n)$ has 9 vertices, 13 edges, and 7 faces.  For small values of $g$ and $p$, the complex $S\T_{\Sigma, 2} / \m$ is much larger.  For example, when $g = 2$ and $p = 1$, this complex has 27 vertices, 110 edges, and 63 faces.  Small values of $g$ and $p$ represent the best case scenario,  $\ST_{\Sigma, 2} / \m$ grows exponentially as either the genus or the number of punctures of $\Sigma$ increases.

\begin{proposition}
Let $(\Sigma, \psi)$ be a marked genus $g$ surface with $p$ punctures where $g \geq 2$ is even and $p \geq 1$ is odd.  There are at least $2^{(p-1)/2 + g/2 + 1}$ vertices in $\ST_{\Sigma, 2} / \m$.
\label{prop:unbound}
\end{proposition}

Proposition \ref{prop:unbound} investigates a subset of the orbits of vertices with underlying graph $R_{2g + p -1}$.  The idea of the proof is that the number of orbits in this subset doubles with each increase of 2 in the genus or the number of punctures of $\Sigma$.  The tools used in the proof of Proposition \ref{prop:unbound} are developed in the following example.

\begin{example}
Let $g = 2$ and $p = 3$ and consider the marking $\psi$ for $\Sigma$ shown in Frame (A) of Figure \ref{fig:torus}, where the symbols $\{ a_1, a_2, a_3, a_4, a_5, a_6 \}$ are the generators of $\pi_1(R_6)$.  Let $id \colon R_6 \to R_6$ be the identity map, and note that $(R_6, id)$ can be drawn in $(\Sigma, \psi)$ by Frame (A) of Figure \ref{fig:torus}.  More formally, the ribbon structure that draws $(R_6, id)$ in $(\Sigma, \psi)$ is
\begin{equation*}
a^i_1, a^t_1, a^i_2, a^t_2, a^i_3, a^i_4, a^t_3, a^t_4, a^i_5, a^i_6, a^t_5, a^t_6,
\end{equation*}
where for a directed edge $e$ the initial half edge is denoted $e^i$ and the terminal half edge is denoted $e^t$.  Let $\rho_j \in Aut(F_6)$ be the automorphism that right multiplies $a_j$ by the inverse of $a_{j+1}$ and leaves all other generators fixed.  Consider the vertex $(R_6, id)\rho_1$; recall that the action of $\rho_1$ on $(R_6, id)$ can be expressed by altering the labeling on $R_6$ by $\rho_1^{-1}$.  Let $e_j$ denote the labelling $a_ja_{j+1}$ and note that $(R_6, id)\rho_1$ can be drawn in $(\Sigma, \psi)$ using the ribbon structure
\begin{equation*}
e^i_1, a^i_2, a^t_2, e^t_1, a^i_3, a^i_4, a^t_3, a^t_4, a^i_5, a^i_6, a^t_5, a^t_6.
\end{equation*}
The ribbon surface for $(R_6, id)\rho_1$ given by this ribbon structure is shown in Frame (B) of Figure \ref{fig:torus}.  Note that $(R_6, id)\rho_4$ can also be drawn in $(\Sigma, \psi)$ using the ribbon structure
\begin{equation*}
a^i_1, a^t_1, a^i_2, a^t_2, a^i_3, e^i_4, a^t_3, a^i_5, a^i_6, a^t_5, e^t_4, a^t_6.
\end{equation*}
The ribbon surface for $(R_6, id)\rho_4$ given by this ribbon structure is shown in Frame (C) of Figure \ref{fig:torus}.  The alterations to the ribbon structure that draws $(R_6, id)$ in $(\Sigma, \psi)$ caused by the actions of $\rho_1$ and $\rho_4$ are independent.  That is, $(R_6, id)\rho_1\rho_4$ can be drawn in $(\Sigma, \psi)$ using the ribbon structure
\begin{equation*}
e^i_1, a^i_2, a^t_2, e^t_1, a^i_3, e^i_4, a^t_3, a^i_5, a^i_6, a^t_5, e^t_4, a^t_6.
\end{equation*}
The ribbon surface for $(R_6, id)\rho_1\rho_4 $ given by this ribbon structure is shown in Frame (D) of Figure \ref{fig:torus}.  

\begin{figure}
\begin{picture}(315, 200)
\put(0,110){\subfloat[]{\fbox{\includegraphics[scale = .35]{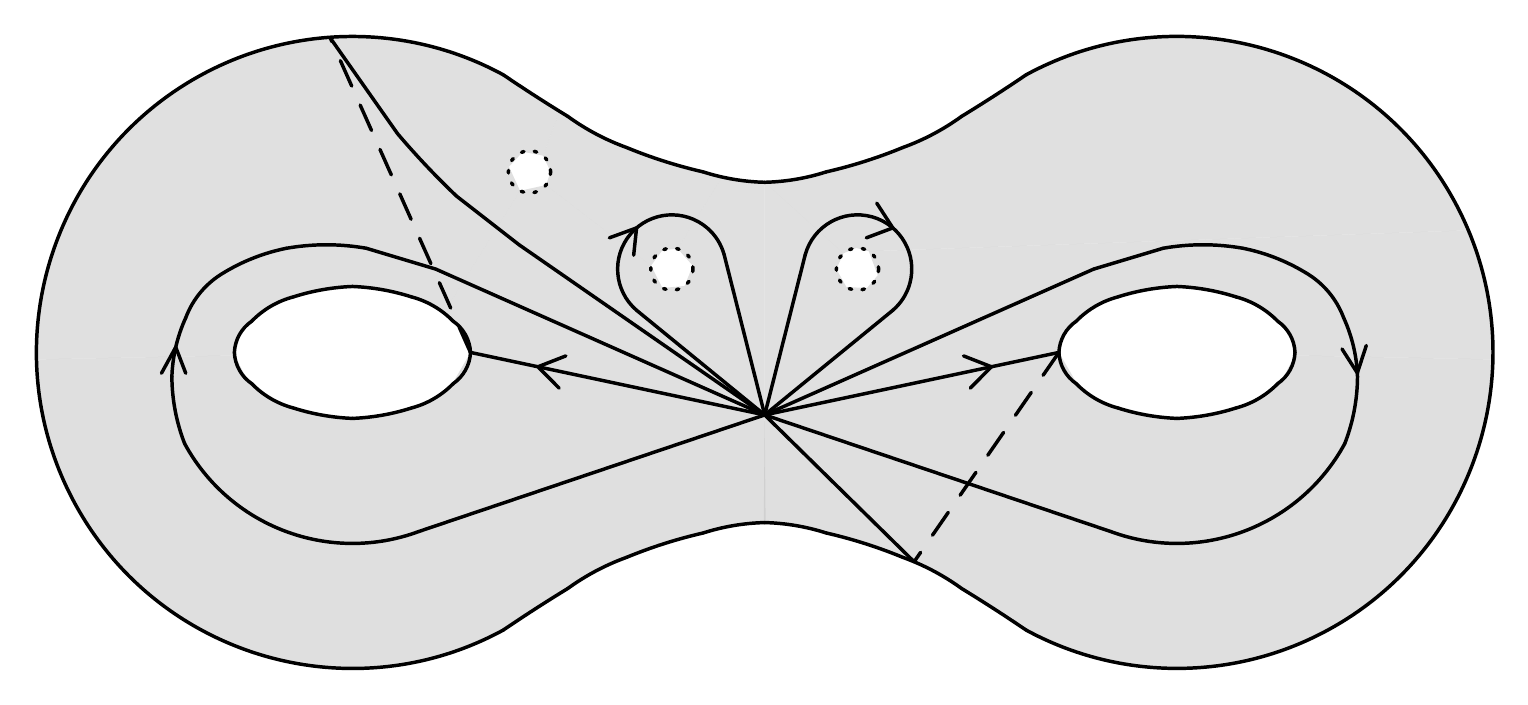}}}}
\put(160,110){\subfloat[]{\fbox{\includegraphics[scale = .35]{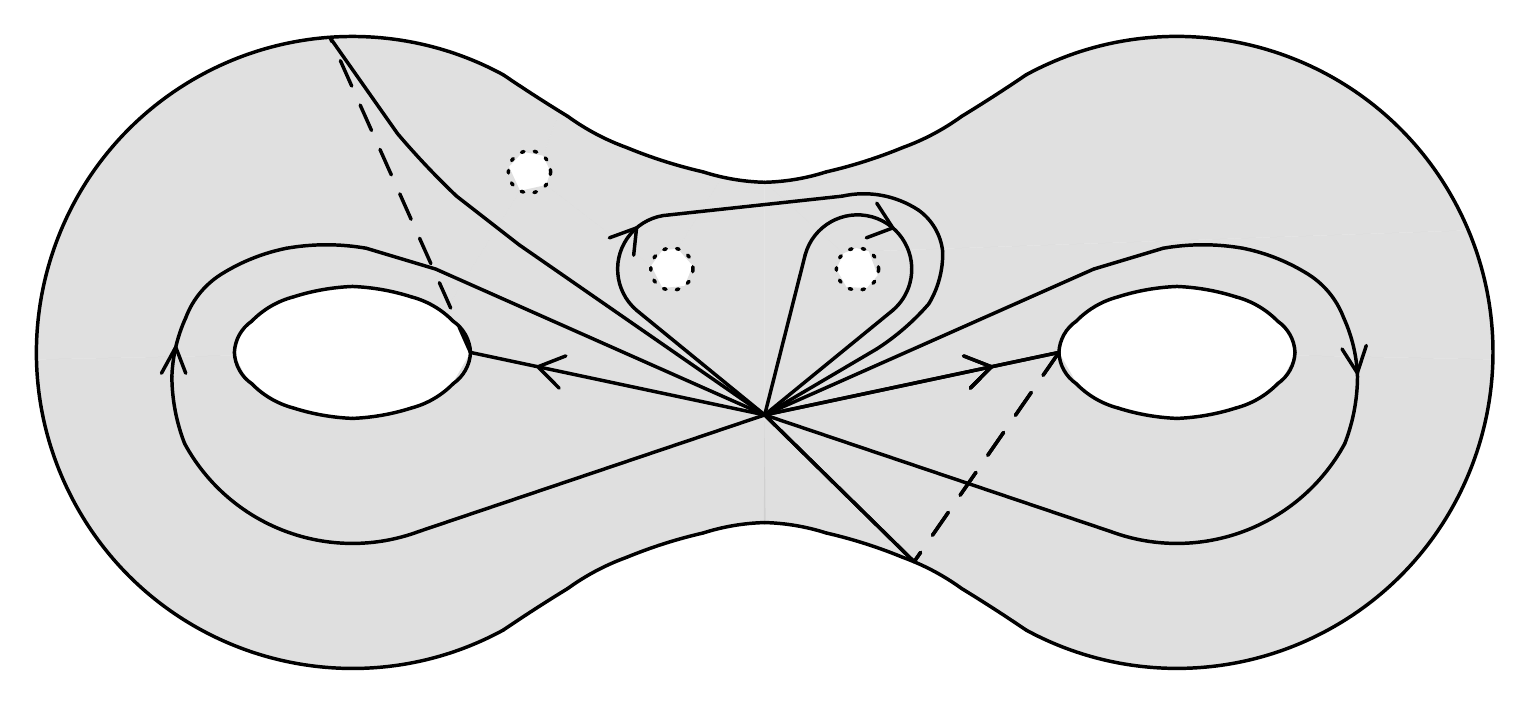}}}}
\put(0,10){\subfloat[]{\fbox{\includegraphics[scale = .35]{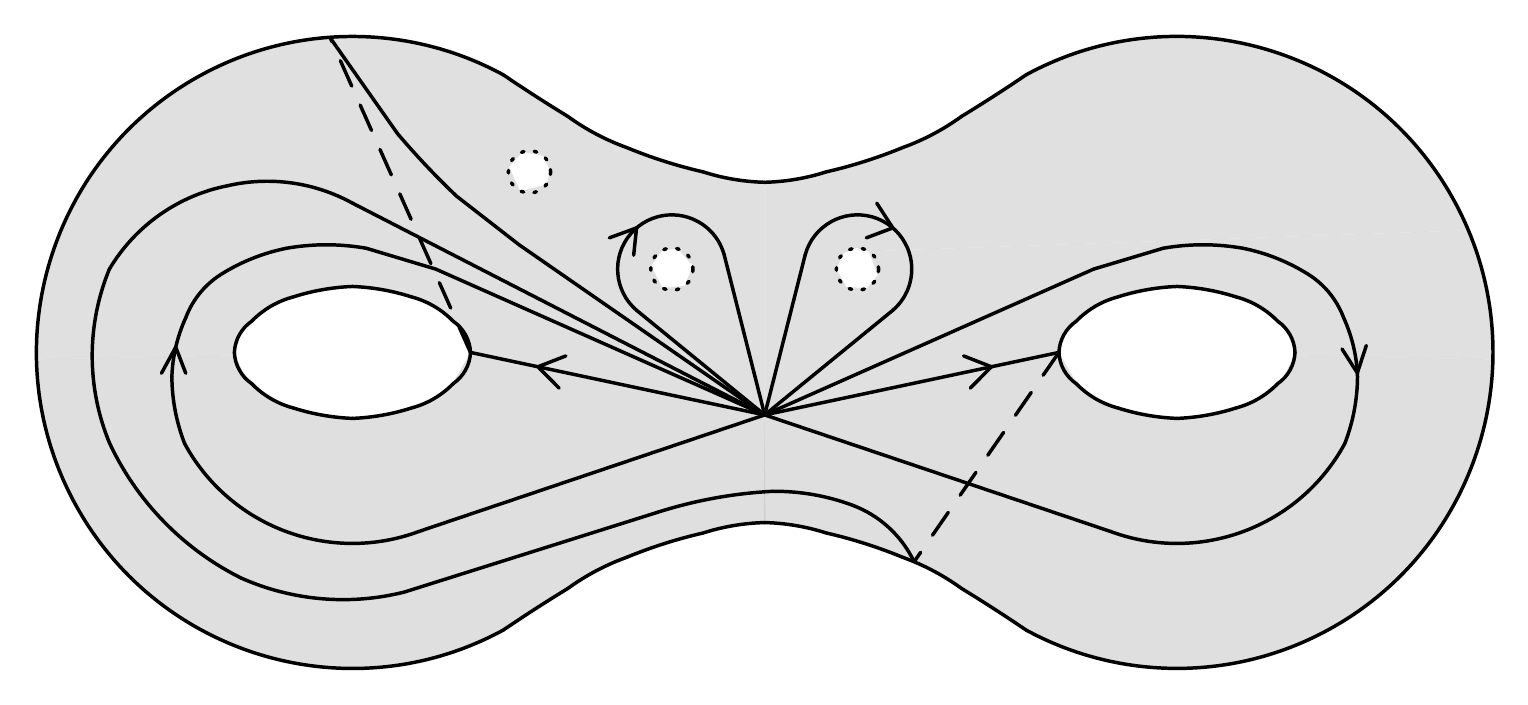}}}}
\put(160,10){\subfloat[]{\fbox{\includegraphics[scale = .35]{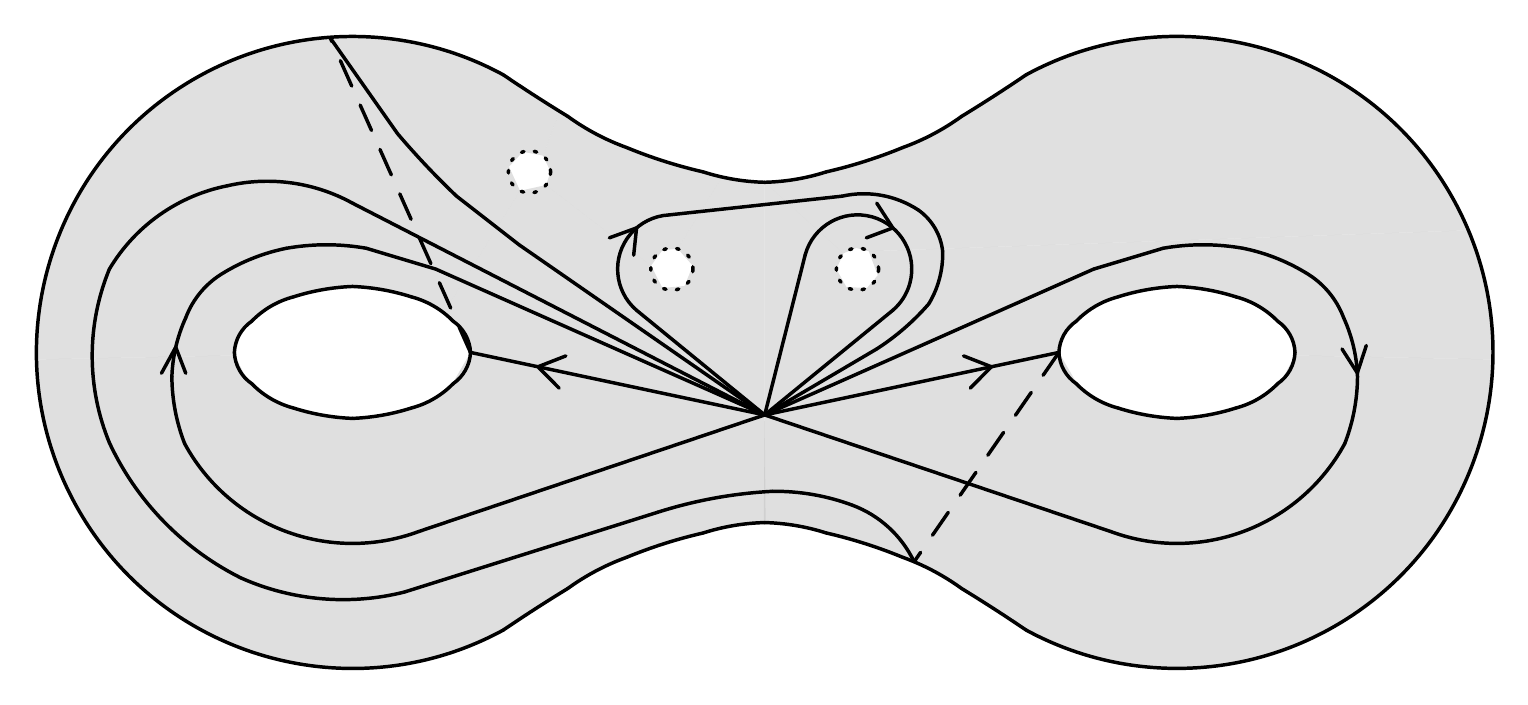}}}}
\put(59,161){$a_1$}
\put(93,161){$a_2$}
\put(120,119){$a_3$}
\put(108,134){$a_4$}
\put(35,119){$a_5$}
\put(53,136){$a_6$}
\put(219,161){$e_1$}
\put(233,149){$a_2$}
\put(280,119){$a_3$}
\put(268,134){$a_4$}
\put(195,119){$a_5$}
\put(213,136){$a_6$}
\put(219,60){$e_1$}
\put(233,49){$a_2$}
\put(279,19){$a_3$}
\put(267,34){$e_4$}
\put(194,32){$a_5$}
\put(212,37){$a_6$}
\put(59,61){$a_1$}
\put(93,61){$a_2$}
\put(119,19){$a_3$}
\put(107,34){$e_4$}
\put(34,32){$a_5$}
\put(52,37){$a_6$}
\end{picture}
\caption{}
\label{fig:torus}
\end{figure}

While the four vertices $(R_6, id)$, $(R_6, id)\rho_1$, $(R_6, id)\rho_4$, and $(R_6, id)\rho_1\rho_4 $ are in the same orbit of $S\A_{6,2}$ under the action of $Aut(F_6)$, they are not in the same orbit of $\m$.  The action of $\m$ on $\ST_{\Sigma}$ does not alter the ribbon structure, only the labeling on the graph, and these three vertices have distinct ribbon structures.  These vertices are representatives of distinct orbits in $S\T_{\Sigma}$ under the action of $\m$.

The effect of $\rho_1$ on the ribbon structure of $(R_6, id)$ is to alter a length four sublist.  Let 
\begin{eqnarray*}
\mathcal{A}_j &=& a^i_j, a^t_j, a^i_{j+1}, a^t_{j+1}, \; \mathrm{and}\\
\mathcal{A}'_j&=& e^i_j, a^i_{j+1}, a^t_{j+1}, e^t_j.  
\end{eqnarray*}
In terms of this notation, $\rho_1$ altered the ribbon structure by replacing $\mathcal{A}_1$ with $\mathcal{A}'_1$.  Similarly, the effect of $\rho_4$ on the ribbon structures of $(R_6, id)$ and $(R_6, id)\rho_1$ is to alter a length eight sublist.  Let 
\begin{eqnarray*}
\mathcal{B}_j &=& a^i_{j-1}, a^i_j, a^t_{j-1}, a^t_j, a^i_{j+1}, a^i_{j+2}, a^t_{j+1}, a^t_{j+2}, \; \mathrm{and}\\
\mathcal{B}'_j &=& a^i_{j-1}, e^i_j, a^t_{j-1}, a^i_{j+1}, a^i_{j+2}, a^t_{j+1}, e^t_j, a^t_{j+2}.  
\end{eqnarray*}
The alteration caused by $\rho_4$ replaces $\mathcal{B}_4$ with $\mathcal{B}'_4$.  We will refer to these sublists $\mathcal{A}_j$, $\mathcal{A}'_j$, $\mathcal{B}_j$, and $\mathcal{B}'_j$ as \emph{blocks}.
These block alteration patterns are the basis of the proof of Proposition  \ref{prop:unbound}.  
\end{example}

\begin{proof}[Proof of Proposition \ref{prop:unbound}]
Note that demonstrating the proposition for a particular marking $\psi$ suffices to prove it for all $\psi$, since $S\T_{\Sigma,2} / \m$ is independent of $\psi$.  We choose $\psi$ shown in Figure \ref{fig:gentorus}.  This figure draws $(R_n, id)$ in $(\Sigma, \psi)$ by the ribbon structure
\begin{equation*}
O = \mathcal{A}_1, \mathcal{A}_3, \ldots, \mathcal{A}_{p-1}, \mathcal{B}_{p+2}, \mathcal{B}_{p+6}, \ldots \mathcal{B}_{n-2},  
\end{equation*} 
where $n = 2g+p-1$.  Consider a subset $C$ of the set of automorphisms $\{ \rho_1$, $\rho_3$, \ldots, $\rho_{p-1}$, $\rho_{p+2}$, $\rho_{p+6}$, \ldots, $\rho_{n-2} \}$.  Let $\sigma_C$ be the product of the automorphisms in $C$. Note that $(R_n, id)\sigma_C \in S\T_{\Sigma}$ as $(R_n, id)\sigma_C$ is drawn in $(\Sigma, \psi)$ by a ribbon structure given by altering the blocks of $O$.  Specifically, for each $\rho_k \in C$ we switch $\mathcal{A}_k$ with $\mathcal{A}'_k$ if $k \leq p-1$ and switch $\mathcal{B}_k$ with $\mathcal{B}'_k$ if $k \geq p+2$, and leave all the other blocks fixed.  For each subset $C$, these ribbon structures are distinct, and so each vertex $(R_n, id)\sigma_C$ in $S\T_{\Sigma, 2}$ represents a distinct orbit under the action of $\m$.  Hence, there are at least $2^{(p-1)/2 + g/2 + 1}$ vertices in $S\T_{\Sigma, 2} / \m$.
\end{proof}

\begin{figure}
\begin{picture}(315, 100)
\put(0,0){\fbox{\includegraphics[scale = .42]{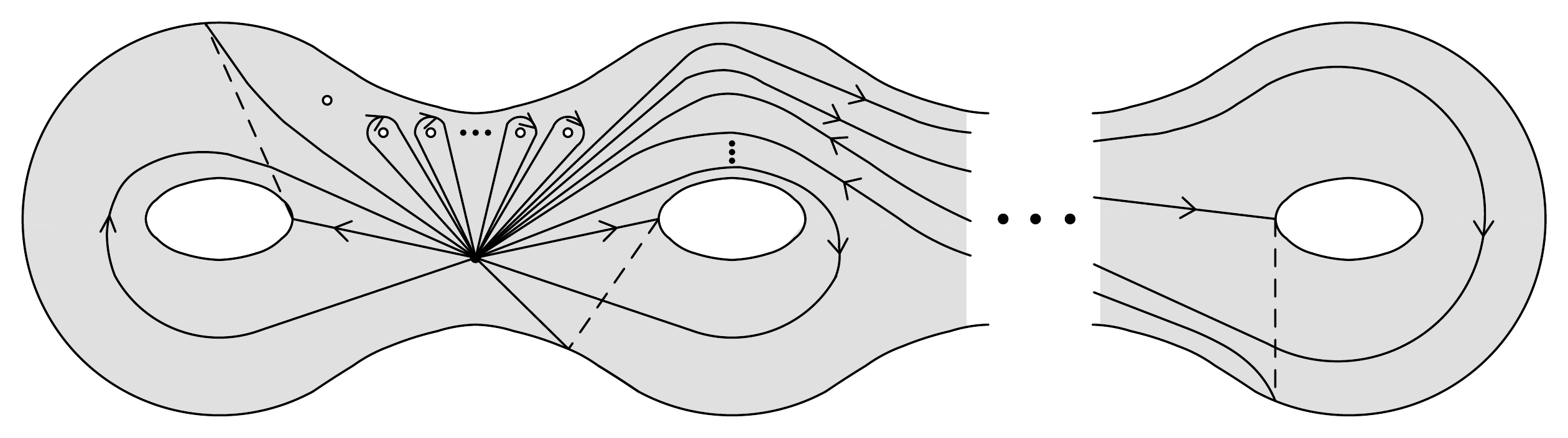}}}
\put(70,74){$a_1$}
\put(80,72){$a_2$}
\put(92,71){$a_{p-1}$}
\put(110,77){$a_p$}
\put(260,60){$a_{p+1}$}
\put(225,52){$a_{p+2}$}
\put(167,30){$a_{n-3}$}
\put(92,12){$a_{n-2}$}
\put(33,12){$a_{n-1}$}
\put(60,34){$a_n$}
\end{picture}
\caption{}
\label{fig:gentorus}
\end{figure}

While the methods of \cite{afv} could be applied to $S\T_{\Sigma, 2} / \m$, the resulting presentation for $\m$ is unlikely to be simpler than existing presentations as the quotient $S\T_{\Sigma, 2} / \m$ is relatively large for small values of $g$ and $p$.  A better approach would be to improve on the simplicial complex $S\T_{\Sigma, 2}$.  That is, does $S\T_{\Sigma, 2}$ contain a simply connected subcomplex that is preserved under the action of $\m$?  This question remains open, a small enough complex could be fruitful in producing a presentation.  More to the point, does $S\T_{\Sigma, k}$ contain a $(k-1)$-connected subcomplex that is preserved under the action of $\m$ and that remains constant as the genus and number of punctures of $\Sigma$ increases?  This stabilization of $S\A_{n,k}$ is the primary property used by Hatcher and Vogtmann in \cite{hatcher-vogtmann}.  While the analogs of Hatcher and Vogtmann's homology results were already shown for mapping class groups by Harer in \cite{harer}, the question of the stabilization of the subcomplexes $S\T_{\Sigma, k}$ remains both open and interesting.

\section*{Acknowledgements}

I thank my advisor Karen Vogtmann.  This work contains part of the results from my Ph.D. thesis, and without Karen's support and insight this work would not have been possible.  I also thank Allen Hatcher for several helpful conversations.  Lastly, I thank the referee, whose comments not only significantly improved the exposition of this work but also improved some of the results. 

\bibliographystyle{plain}

\end{document}